\documentclass[12pt, openany, letterpaper]{amsart}
\usepackage[margin=1.2in]{geometry}
\usepackage{amsmath, amssymb, amsthm, amsfonts}
\usepackage{enumitem}
\usepackage{hyperref}
\usepackage{cleveref}
\usepackage{booktabs}
\usepackage{mathrsfs}
\usepackage{tikz-cd}

\theoremstyle{plain}

\newtheorem*{theorem*}{Theorem}
\newtheorem{theorem}{Theorem}[section]
\newtheorem{proposition}[theorem]{Proposition}
\newtheorem{lemma}[theorem]{Lemma}
\newtheorem{corollary}[theorem]{Corollary}

\theoremstyle{definition}
\newtheorem{definition}[theorem]{Definition}
\newtheorem{example}[theorem]{Example}

\theoremstyle{remark}
\newtheorem{remark}[theorem]{Remark}
\newtheorem{notation}[theorem]{Notation}

\newcommand{\F}{\mathbb{F}}

\newcommand{\Pp}{\mathbb{P}}

\newcommand{\OO}{\mathcal{O}}
\newcommand{\LL}{\mathcal{L}}
\newcommand{\MM}{\mathcal{M}}
\newcommand{\BB}{\mathcal{B}}
\newcommand{\FF}{\mathcal{F}}
\newcommand{\GG}{\mathcal{G}}

\newcommand{\two}{ \mathbf{I\!I}}
\newcommand{\one}{  \mathbf{I}}

\newcommand{\f}{\mathfrak{f}}
\newcommand{\p}{\mathfrak{p}}

\newcommand{\prI}{\operatorname{Pr}_\one}
\newcommand{\prII}{\operatorname{Pr}_{\two}}
\newcommand{\m}{\mathfrak{m}}

\newcommand{\OmegaOne}{\Omega^1}
\newcommand{\omegaA}{\Omega_A^1}
\newcommand{\Weil}{\operatorname{Weil}}
\newcommand{\Cinf}{\mathbb{C}_\infty}

\newcommand{\res}{\operatorname{res}}
\newcommand{\Res}{\operatorname{Res}}
\newcommand{\Tr}{\operatorname{Tr}}
\newcommand{\pair}[2]{\langle #1, #2 \rangle}

\newcommand{\PP}{\operatorname{PP}}
\newcommand{\tPP}{\widetilde{\operatorname{PP}}}
\newcommand{\Hom}{\operatorname{Hom}}
\newcommand{\End}{\operatorname{End}}

\newcommand{\CasFin}{\operatorname{Cas}_{\f}^{\text{fin}}}
 
\newcommand{\WeilOp}{\operatorname{O}_{\omega,\f}^{[2]}}

\newcommand{\WeilOpr}{\operatorname{O}_{\omega,\f}^{[r]}}

\newcommand{\id}{\operatorname{id}}

\newcommand{\Df}{D_{\f}}

\newcommand{\supp}[1]{|#1|}

\newcommand{\DeltaDiag}{\Delta}

\newcommand{\UpsilonFrame}{\Upsilon_e}

\newcommand{\inc}{\operatorname{inc}}

\begin{document}
\title{Geometric Realization of Finite Residue Casimirs and Weil Operators via Szeg\H{o} Kernels}

\author{Chuangqiang Hu}
\address{School of Mathematics, Sun Yat-Sen University, Guangzhou, 510275,
P. R. China}
\email{\href{huchq5@mail.sysu.edu.cn}{huchq5@mail.sysu.edu.cn}}

\author{Lishan Yu}
\address{Beijing Institute of Mathematical Sciences and Applications, Beijing, 101408, P. R. China}
\email{\href{lishany@bimsa.cn}{lishany@bimsa.cn}}

\begin{abstract}
For a smooth projective curve over a finite field with a fixed point at infinity, we establish a universal correspondence between finite residue duality and geometric kernel functions. We prove that the finite residue Casimir tensor is realized geometrically as the intrinsic principal part of the normalized Szeg\H{o} kernel for any acyclic line bundle, and equivalently that this kernel acts as a reproducing kernel for the finite residue pairing, in exact analogy with the classical Cauchy integral formula. In the polynomial case these equivalent descriptions yield a closed formula involving the rank-two Weil operator, identified with the classical divided difference, recovering a remainder identity of Hu--Ou. These results provide a geometric foundation for the study of Anderson generating functions and the Weil pairing for Drinfeld modules.
\end{abstract}

\maketitle

\tableofcontents

\vspace{1em}

\section{Introduction}
\label{sec:introduction}
The Weil pairing for Drinfeld modules, introduced in Drinfeld's foundational theory of elliptic modules \cite{DVG74}, is a cornerstone of function-field arithmetic, serving as the natural analogue of the classical Weil pairing on elliptic curves and abelian varieties; a general background on Drinfeld modules may be found in \cite{Pap23, PB21, Tha04}. Unlike its elliptic-curve counterpart, which is a bilinear alternating form taking values in roots of unity, the Weil pairing for Drinfeld modules is a multilinear alternating map
\[
\Weil_{\f} : \prod_{i=1}^r \phi[\f] \longrightarrow \left(\wedge^r \phi\right)[\f],
\]
where $\phi$ is a rank-$r$ Drinfeld $\F_q[t]$-module and $\f$ is an ideal of definition. The existence of this pairing was first established by van der Heiden \cite{vdHGJ04} using the deep machinery of Anderson $t$-motives \cite{AGW86}. Subsequently, Papikian \cite{Pap23} obtained an explicit formula for the rank-two Weil pairing without invoking Anderson $t$-motives. Though elementary, his treatment is mainly computational and does not provide a geometric motivation for the existence of the Weil operator. Katen \cite{KJ21} later extended Papikian's formula to arbitrary ranks, revealing a recursive structure for the higher-rank Weil operators $\WeilOpr$.

A central insight into the nature of these operators came from the work of Hu and Ou-Yang \cite{HuOuY26}, who established a fundamental connection between Weil operators and the $f$-remainder of Anderson generating functions. Related algebraic identities for Anderson generating functions had previously been studied by El-Guindy and Papanikolas \cite{ElP14}, and their arithmetic significance for Drinfeld torsion extensions was further explored by Maurischat and Perkins \cite{MP22}. A key lemma in their work states that the $f$-remainder of the simple rational function $1/(\theta - t)$ is precisely
\begin{equation}\label{eq:intro}
\left[\frac{1}{\theta - t}\right]_{\f} = \frac{1}{f(\theta)} \cdot \frac{f(\theta) - f(t)}{\theta - t} = \frac{\WeilOp(\theta, t)}{f(\theta)}.
\end{equation}
This identity demonstrates that the rank-two Weil operator is not an ad hoc combinatorial construction, but rather the polynomial kernel that governs the reduction of Cauchy-type kernels modulo the ideal $\f$. From this perspective, the Weil operator emerges naturally from the duality between functions and differentials on the affine line, grounding its algebraic definition in intrinsic geometric properties of the base curve.

Over the complex numbers, the Szeg\H{o} kernel occupies a central position in the geometry of algebraic curves, intimately connected to theta functions, determinant line bundles and the representation theory of loop groups \cite{BenZviBiswas2003}. In this setting, it serves as a universal reproducing kernel whose singular behavior along the diagonal encodes the fundamental duality between functions and differentials.

By contrast, in the function-field setting, despite the deep arithmetic significance of Drinfeld modules and their torsion points, a systematic geometric theory of Szeg\H{o} kernels and their relation to residue duality has been lacking. The goal of the present work is to fill this gap by \textbf{transplanting the characteristic-zero geometric framework of Szeg\H{o} kernels into function-field arithmetic}. Drawing on the classical foundations of function-field arithmetic \cite{Goss96,Pap23,Tha04}, we show that the finite residue Casimir---the canonical element representing the identity map under the perfect residue pairing between $A/\f$ and its differential dual---is realized geometrically as the intrinsic principal part of the normalized Szeg\H{o} kernel for any acyclic line bundle along the finite divisor $\Df$. 

Our main result can be stated as follows. Let $X/\F_q$ be a smooth projective curve with a fixed closed point $\infty$, and let $A=\Gamma(X\setminus\{\infty\},\OO_X)$. For any nonzero ideal $\f\subseteq A$, we denote by $\CasFin$ the Casimir tensor, and by $\inc$ its scalar extension to $\Cinf$. Let $\LL$ be an acyclic line bundle on the base-changed curve $X_{\Cinf}$, and let $S_{\LL}$ be its normalized Szeg\H{o} kernel (see Section \ref{sec:acyclic-szego}). Denote by $\PP_D$ the intrinsic principal part map along the divisor $D=\Df\otimes_{\F_q}\Cinf$, and by $\UpsilonFrame$ the frame identification induced by any local frame $e$ of $\LL$ near the support of $D$.

\begin{theorem*}[Finite Residue--Cauchy Theorem]
With the notation above, the principal part of the normalized Szeg\H{o} kernel along $D$ realizes the scalar extension of the finite residue Casimir; that is,
\begin{equation}\label{eq:main-intro}
\UpsilonFrame\bigl(\PP_D(S_{\LL})\bigr)=\inc\bigl(\CasFin\bigr).
\end{equation}
In particular, the left-hand side of \eqref{eq:main-intro} is independent of the choice of acyclic line bundle $\LL$, local frame $e$, and local coordinates.
\end{theorem*}

This result establishes a \textbf{universal correspondence between finite residue duality and geometric kernel functions}. It is equivalent to the following \emph{reproducing property}, which is the exact function-field analogue of the classical Cauchy integral formula: just as contour integration against the Cauchy kernel recovers a holomorphic function from its boundary values, contraction of our kernel against any local section through the residue pairing returns the restriction of that section to $D$.

\begin{theorem*}[Reproducing property]
Let $\LL$ be acyclic and let $\tilde\alpha\in H^0(U,\LL)$ be a section defined on a neighbourhood $U\supseteq\supp D$. Then
\begin{equation}\label{eq:reproducing-intro}
\Bigl(\id\otimes\bigl\langle\tilde\alpha|_D,\cdot\bigr\rangle_D\Bigr)\bigl(\PP_D(S_{\LL})\bigr)
=\sum_{y\in\supp D}\res_y^{\two}\bigl(S_{\LL}\cdot\tilde\alpha(\two)\bigr)
=\tilde\alpha|_D.
\end{equation}
In particular, $\PP_D(S_{\LL})$ is the identity tensor for the geometric residue pairing on $D$, and hence encodes exactly the same universal duality as \eqref{eq:main-intro}.
\end{theorem*}

In the polynomial case $A=\F_q[t]$, the Szeg\H{o} kernel reduces to the familiar Cauchy kernel $dt_{\two}/(t_{\two}-t_{\one})$, and both \eqref{eq:main-intro} and \eqref{eq:reproducing-intro} recover exactly the remainder identity \eqref{eq:intro} of Hu--Ou \cite{HuOuY26}. For a general Drinfeld coefficient ring, the geometry forces us to replace polynomial functions with K\"ahler differentials, and the Szeg\H{o} kernel becomes a section of a nontrivial line-bundle-valued sheaf on the product curve; nevertheless, the universal correspondence persists. These results apply to Drinfeld modules attached to higher-genus curves, notably elliptic curves, which have long been studied in the literature \cite{BK95,DH94,GP18}.
In the polynomial case $A=\F_q[t]$, the Szeg\H{o} kernel reduces to the familiar Cauchy kernel $dt_{\two}/(t_{\two}-t_{\one})$, and \eqref{eq:main-intro} recovers exactly the remainder identity \eqref{eq:intro} of Hu--Ou \cite{HuOuY26}. For a general Drinfeld coefficient ring, the geometry forces us to replace polynomial functions with K\"ahler differentials, and the Szeg\H{o} kernel becomes a section of a nontrivial line-bundle-valued sheaf on the product curve; nevertheless, the universal correspondence \eqref{eq:main-intro} persists. These results apply to Drinfeld modules attached to higher-genus curves, notably elliptic curves, which have long been studied in the literature \cite{BK95,DH94,GP18}.

Four features of our treatment deserve special emphasis:
\begin{enumerate}
\item \textbf{Intrinsic principal parts.} We define the principal part of a section along a finite divisor intrinsically, as the class of its restriction to $D\times U$ in a suitable quotient space. Independence from local coordinates and frames becomes a formal consequence of functoriality, eliminating the need for \emph{a posteriori} verification.

\item \textbf{Elimination of circularity.} We first establish the frame-coefficient version of the main theorem (Proposition \ref{prop:frame-coeff-pp}), and then deduce the intrinsic, frame-independent statement from the functoriality of the principal part map. This logical order removes a circularity present in earlier treatments.

\item \textbf{Minimal arithmetic input.} The bridge between the $\F_q$-arithmetic of the Casimir tensor and the $\Cinf$-geometry of the Szeg\H{o} kernel is a single base-change lemma (Lemma \ref{lemma:base-change-residues}); everything else is pure geometry over the algebraically closed field $\Cinf$.

\item \textbf{Universal higher-genus correspondence.} The framework extends seamlessly from $\Pp^1$ to arbitrary higher-genus curves. Even though the Szeg\H{o} kernel is no longer an elementary rational function in genus $g\ge 1$, its principal part consistently realizes the universal finite-residue Casimir tensor. In particular, our geometric kernel $C_X(\cdot,\cdot)$ recovers the $G_u$- and $J_u$-functions appearing in Anderson generating function constructions on elliptic curves \cite{GP18}.
\end{enumerate}

We now outline the structure of the paper.

In Section \ref{sec:geometric-setting}, we fix our geometric notation, recall the definition of local and finite residues, and prove the base-change compatibility of residues. We also state the Serre duality theorem for finite divisors, which is the geometric ancestor of all our duality statements.

In Section \ref{sec:algebraic-side}, we introduce the two fundamental modules $V_{\f}$ and $V_{\f}^{\dagger}$, define the finite residue pairing between them, and prove that it is perfect. We then define the finite residue Casimir tensor, prove its basis independence, and derive its split form after scalar extension to $\Cinf$.

In Section \ref{sec:acyclic-szego}, we review the notion of an acyclic line bundle, prove their existence in every genus, and construct the normalized Szeg\H{o} kernel as the unique section of the Szeg\H{o} bundle with diagonal residue $1$. We also establish the transformation law of Szeg\H{o} kernels under twisting by a line bundle that is trivial near the divisor.

In Section \ref{sec:intrinsic-pp}, we present the technical heart of the paper: the intrinsic definition of the principal part map. We derive its basic properties---exact truncation, regularity, functoriality, and frame independence---and introduce the transposed principal part.

In Section \ref{sec:finite-residue-cauchy}, we state and prove our main theorem: the principal part of the normalized Szeg\H{o} kernel along any finite divisor equals the scalar extension of the finite residue Casimir. We also derive the reproducing property of the kernel, and specialize to the polynomial ring to obtain the closed formula involving the divided difference.

Finally, we define the Weil operator as the scalar part of the Casimir, and prove its local decomposition theorem: it is the Chinese-remainder assembly of local divided differences, with trace-Casimir factors at non-rational points. We also give two proofs of its symmetry.

Sections \ref{sec:ex-projective-line}, \ref{sec:ex-hhy}, and \ref{sec:ex-elliptic} are devoted to examples. We treat the projective line (including a non-split quadratic modulus), the Hu-Huang-Yau ring (exhibiting the phenomenon of non-principal differentials), and elliptic curves. In each case we verify the main theorem explicitly.

\section*{Notation List}


We fix the following notation throughout the paper, for ease of reference.

\vspace{0.5em}
\begin{itemize}[leftmargin=*, itemsep=0.3em]
\item $q$: A power of $p$, defining the size of the finite base field.
\item $\F_q$: The finite field with $q$ elements.
\item $\Cinf$: The completed algebraic closure of the local field $\F_q((t_\infty))$, the function-field analogue of the complex numbers.
\item $X$: A smooth, projective, geometrically connected curve defined over $\F_q$.
\item $g$: The genus of the curve $X$.
\item $\infty$: A distinguished closed point on $X$, called the point at infinity.
\item $U$: The affine open subset $X \setminus \{\infty\}$.
\item $\OO_X$: The structure sheaf of the curve $X$.
\item $\OmegaOne_X$: The sheaf of K\"ahler differentials on $X$ over $\F_q$.
\item $X_{\Cinf}$: The base-changed curve $X \times_{\F_q} \Cinf$.
\item $A = \Gamma(U, \OO_X)$: The Drinfeld coefficient ring, a Dedekind domain.
\item $K = \F_q(X)$: The function field of $X$, the fraction field of $A$.
\item $d_\infty = \deg(\infty) = [\kappa(\infty) : \F_q]$: The degree of the point at infinity.
\item $\f \subseteq A$: A nonzero ideal in the coefficient ring.
\item $\Df$: The effective divisor on $U$ associated to the ideal $\f$.
\item $\supp{\Df}$: The support of the divisor $\Df$.
\item $D = D_{\f, \Cinf}$: The base change of $\Df$ to $X_{\Cinf}$.
\item $m_y$: The multiplicity of a point $y$ in the divisor $D$.
\item $V_{\f} = A / \f$: The quotient algebra of functions modulo $\f$.
\item $V_{\f}^{\dagger} = \f^{-1} \omegaA / \omegaA$: The dual differential quotient module.
\item $\res_x$: The local residue operator at a closed point $x$.
\item $\Res_{\Df}$: The finite residue operator, summing traces of local residues over $\supp{\Df}$.
\item $\pair{\cdot}{\cdot}_{\f}$: The finite residue pairing $V_{\f} \times V_{\f}^{\dagger} \to \F_q$.
\item $\CasFin$: The finite residue Casimir tensor in $V_{\f} \otimes_{\F_q} V_{\f}^{\dagger}$.
\item $\LL$: A line bundle on $X_{\Cinf}$.
\item $H^i(X, \FF)$: The $i$-th sheaf cohomology group of the sheaf $\FF$.
\item $\prI, \prII$: The two projection maps from $X_{\Cinf} \times X_{\Cinf}$ to $X_{\Cinf}$.
\item $\DeltaDiag$: The diagonal subscheme of $X_{\Cinf} \times X_{\Cinf}$.
\item $\BB_{\LL}$: The Szeg\H{o} bundle on the product curve.
\item $S_{\LL}$: The normalized Szeg\H{o} kernel, a section of $\BB_{\LL}(\DeltaDiag)$.
\item $W_D(\LL)$: The space of principal parts along the divisor $D$.
\item $\PP_D$: The principal part projection map along $D$.
\item $e$: A local frame (nowhere-vanishing section) for a line bundle.
\item $\UpsilonFrame$: The frame-induced identification map on principal parts.
\item $\inc$: The canonical scalar extension map from $\F_q$ to $\Cinf$.
\item $\tPP_D$: The transposed principal part map.
\item $\omega_{\f}$: A differential generator of $V_{\f}^{\dagger}$, typically $\omega / f$.
\item $\iota_{\omega_{\f}}$: The isomorphism $V_{\f} \to V_{\f}^{\dagger}$ given by multiplication by $\omega_{\f}$.
\item $\WeilOp$: The rank-two Weil operator in $V_{\f} \otimes_{\F_q} V_{\f}$.
\item $\mathcal{T}$: The flip operator on tensor products, $\mathcal{T}(a \otimes b) = b \otimes a$.
\item $ \mathcal{T}_\omega $ : The $ \omega$-flip $ W_D(\OO) \to \widetilde{W}_D(\OO), a(\one) \otimes [\gamma](\two) \mapsto [\gamma / \omega](\one) \otimes a \omega(\two)$.
\end{itemize}

\vspace{1em}

\section{Geometric Setting}
\label{sec:geometric-setting}

In this section we fix the basic geometric objects that will be used throughout the paper: the curve, its function field and differential sheaf, the base change to the completed algebraic closure $\Cinf$, and the notion of residue at a closed point. We also prove the fundamental base-change lemma for residues, which is the only genuinely arithmetic ingredient in our main theorem.

\subsection{The Curve and Its Differential Sheaf}
\label{subsec:curve-differential}

Let $X / \F_q$ be a smooth, projective, geometrically connected curve of genus $g$. We fix once and for all a closed point $\infty \in X$, which we refer to as the \textbf{point at infinity}. The complement
\[
U = X \setminus \{\infty\}
\]
is then an affine open subset of $X$. Its ring of global sections
\[
A = \Gamma(U, \OO_U)
\]
is a Dedekind domain, called the \textbf{Drinfeld coefficient ring} associated to the pair $(X, \infty)$. Its fraction field is the function field $K = \F_q(X)$ of the curve.

Let $K_{\infty}$ denote the completion of $K$ with respect to the valuation at $\infty$, and let $\Cinf$ be a completed algebraic closure of $K_{\infty}$. The field $\Cinf$ plays the role of the complex numbers in function-field arithmetic; it is algebraically closed and complete with respect to a non-Archimedean absolute value extending the one on $K_{\infty}$.

We denote by
\[
X_{\Cinf} = X \times_{\F_q} \Cinf, \qquad U_{\Cinf} = U \times_{\F_q} \Cinf
\]
the base changes of $X$ and $U$ to $\Cinf$. The curve $X_{\Cinf}$ is a smooth projective curve over the algebraically closed field $\Cinf$, and $U_{\Cinf}$ is an affine open subset.

Let $\OmegaOne = \OmegaOne_{X / \F_q}$ denote the sheaf of K\"ahler differentials on $X$ relative to $\F_q$. By smoothness, $\OmegaOne$ is a line bundle on $X$, the canonical bundle. By base change, the sheaf $\OmegaOne_{X_{\Cinf} / \Cinf}$ on $X_{\Cinf}$ is identified with the pullback of $\OmegaOne$; on the affine open $U_{\Cinf}$ this identification is an equality. We will therefore write simply $\OmegaOne$ for the differential sheaf, leaving the base implicit.

A closed point $x \in |X|$ has residue field $\kappa(x) = \OO_{X,x} / \m_x$, a finite extension of $\F_q$. Its degree is
\[
\deg(x) = [\kappa(x) : \F_q].
\]
When base-changed to $\Cinf$, a closed point $x$ of degree $s$ splits into $s$ distinct closed points $y_1, \dots, y_s \in |X_{\Cinf}|$, in bijection with the field embeddings $\sigma_j : \kappa(x) \hookrightarrow \Cinf$. If $m_x$ denotes the multiplicity of $x$ in some effective divisor on $X$, then each split point $y_j$ inherits the same multiplicity: $m_{y_j} = m_x$.

Now let $\f \subseteq A$ be a nonzero ideal. Since $A$ is a Dedekind domain, $\f$ factors uniquely as a product of prime ideals. Associated to $\f$ is an effective divisor $\Df$ on $U$, defined by
\begin{equation}
\label{eq:ideal-divisor}
\f \, \OO_U = \OO_U(-\Df).
\end{equation}
Concretely, if $\f = \prod_{i=1}^r \p_i^{m_i}$ is the prime factorization, then
\[
\Df = \sum_{i=1}^r m_i \, x_i,
\]
where $x_i \in U$ is the closed point corresponding to the prime ideal $\p_i$. We write $\supp{\Df} = \{x_1, \dots, x_r\}$ for the support of the divisor.

We denote by
\[
D = D_{\f, \Cinf}
\]
the base change of $\Df$ to $X_{\Cinf}$. Its support $\supp{D}$ consists of all points of $X_{\Cinf}$ lying over $\supp{\Df}$. Since $D$ is supported on the affine part $U_{\Cinf}$, we may fix an affine open neighbourhood of $\supp{D}$ once and for all; for simplicity we take this neighbourhood to be $U_{\Cinf}$ itself.

\subsection{Residues at Closed Points}
\label{subsec:residues}

We now recall the definition of the residue of a meromorphic differential at a closed point, and establish its behavior under base change
, see \cite{S09,VS06} for more details.

Let $x \in |X|$ be a closed point, and let $v_x$ be a uniformizer at $x$, i.e., a generator of the maximal ideal $\m_x \subseteq \OO_{X,x}$. The completed local ring $\widehat{\OO}_{X,x}$ is isomorphic to the power series ring $\kappa(x)[[v_x]]$. Any rational differential $\omega$ on $X$ can be expanded near $x$ as a Laurent series
\[
\omega = f \, dv_x, \qquad f = \sum_{n \gg -\infty} c_n v_x^n, \quad c_n \in \kappa(x).
\]

\begin{definition}[Local residue]
The \textbf{local residue} of $\omega$ at $x$ is the coefficient of $v_x^{-1}$ in the Laurent expansion of $f$:
\begin{equation}
\label{eq:local-residue-def}
\res_x(\omega) = c_{-1} \in \kappa(x).
\end{equation}
It is a standard fact that this definition is independent of the choice of uniformizer $v_x$.
\end{definition}

Since the residue lies in the residue field $\kappa(x)$, which may be a nontrivial extension of $\F_q$, we define the \textbf{finite residue} by composing with the field trace:
\begin{equation}
\label{eq:finite-residue-point}
\Res_x(\omega) = \Tr_{\kappa(x)/\F_q} \left( \res_x(\omega) \right) \in \F_q.
\end{equation}
More generally, for an effective divisor $D_{\f}$ supported on $U$, we define the total finite residue along $D_{\f}$ by summing over its support:
\begin{equation}
\label{eq:finite-residue-divisor}
\Res_{\Df}(\omega) = \sum_{x \in \supp{\Df}} \Res_x(\omega) = \sum_{x \in \supp{\Df}} \Tr_{\kappa(x)/\F_q} \left( \res_x(a\omega) \right).
\end{equation}

Over the base-changed curve $X_{\Cinf}$, every closed point $y$ has residue field $\Cinf$, so the local residue $\res_y$ already takes values in the field of scalars, and no trace is needed: $\Res_y = \res_y$.

The following lemma relates residues over $\F_q$ to residues over $\Cinf$. It is the arithmetic bridge between the algebraic Casimir and the geometric Szeg\H{o} kernel.

\begin{lemma}[Base change for residues]
\label{lemma:base-change-residues}
Let $x \in |X|$ be a closed point of degree $s$, splitting into points $y_1, \dots, y_s \in |X_{\Cinf}|$ with corresponding embeddings $\sigma_j : \kappa(x) \hookrightarrow \Cinf$. For any rational differential $\omega$ on $X$, let $\omega_{\Cinf}$ denote its base change to $X_{\Cinf}$. Then
\[
\res_{y_j} \left( \omega_{\Cinf} \right) = \sigma_j \left( \res_x(\omega) \right)
\]
for each $j$, and consequently
\begin{equation}
\label{eq:base-change-residue-sum}
\sum_{j=1}^s \res_{y_j} \left( \omega_{\Cinf} \right) = \Res_x(\omega).
\end{equation}
\end{lemma}

\begin{proof}
Choose a uniformizer $v_x$ at $x$, and write $\omega = f \, dv_x$ near $x$ with $f \in \kappa(x)((v_x))$. Under base change, the element $f \otimes 1 \in \kappa(x)((v_x)) \otimes_{\kappa(x)} \Cinf$ maps to the tuple $(\sigma_j(f))_j$ in the product $\prod_{j=1}^s \Cinf((v_x))$, where each $\sigma_j(f)$ is the Laurent series obtained by applying $\sigma_j$ coefficientwise.

The residue operation extracts the coefficient of $v_x^{-1}$. This operation commutes with the field embedding $\sigma_j$, since it acts coefficientwise. Therefore
\[
\res_{y_j}(\omega_{\Cinf}) = \text{coefficient of } v_x^{-1} \text{ in } \sigma_j(f) = \sigma_j \left( \res_x(\omega) \right).
\]
Summing over all $j$, we obtain
\[
\sum_{j=1}^s \res_{y_j}(\omega_{\Cinf}) = \sum_{j=1}^s \sigma_j \left( \res_x(\omega) \right) = \Tr_{\kappa(x)/\F_q} \left( \res_x(\omega) \right) = \Res_x(\omega),
\]
where the second equality is the definition of the field trace as the sum of embeddings. This proves \eqref{eq:base-change-residue-sum}.
\end{proof}

\section{Serre Duality and Casimir Tensor}
\label{sec:algebraic-side}

In this section we define the two finite-dimensional $\F_q$-vector spaces that are the main algebraic objects of the paper, and the canonical perfect residue pairing between them. We then introduce the Casimir tensor of this pairing, which is the algebraic ancestor of the Weil operator.

\subsection{Serre Duality on a Finite Divisor}
\label{subsec:serre-duality-divisor}

Let \(D\) be an effective divisor on \(U\) (over \(\F_q\)).  
For any line bundle \(\mathcal{M}\) on \(U\), consider the two finite-dimensional
\(\F_q\)-vector spaces
\[
H^0(D,\mathcal{M}|_D)
\qquad\text{and}\qquad
H^0\left(D,\frac{(\mathcal{M}^{-1}\otimes\OmegaOne)(D)}
{\mathcal{M}^{-1}\otimes\OmegaOne}\right).
\]
We define the \textbf{finite residue pairing}
\[
\langle \cdot,\cdot\rangle_D:
H^0(D,\mathcal{M}|_D)
\times
H^0\left(D,\frac{(\mathcal{M}^{-1}\otimes\OmegaOne)(D)}
{\mathcal{M}^{-1}\otimes\OmegaOne}\right)
\longrightarrow \F_q
\]
by
\[
\langle m,\xi\rangle_D
=
\sum_{x\in \supp D}
\Tr_{\kappa(x)/\F_q}
\left(
\res_x(m\cdot\tilde{\xi})
\right),
\]
where \(m\) is a section representing the class in \(H^0(D,\mathcal{M}|_D)\),
and \(\tilde{\xi}\) is a section of \((\mathcal{M}^{-1}\otimes\OmegaOne)(D)\)
representing the class in the quotient.

\begin{lemma}[Well-definedness]
The finite residue pairing \(\langle \cdot,\cdot\rangle_D\) is well-defined and
\(\F_q\)-bilinear.
\end{lemma}

\begin{proof}
We must check that the value does not depend on the chosen representatives.

\noindent
\textbf{Independence of \(m\).}
If \(m'\) and \(m\) differ by a section vanishing to order at least
\(\operatorname{ord}_x(D)\) at every \(x\in\supp D\), then the product
\((m'-m)\cdot\tilde{\xi}\) is a regular section of
\(\mathcal{M}^{-1}\otimes\OmegaOne\) near \(D\).
Hence its residue at every point of \(\supp D\) is zero.

\noindent
\textbf{Independence of \(\xi\).}
If \(\tilde{\xi}'\) and \(\tilde{\xi}\) differ by a regular section of
\(\mathcal{M}^{-1}\otimes\OmegaOne\), then \(m\cdot(\tilde{\xi}'-\tilde{\xi})\)
is regular near \(D\), so again all residues vanish.

Bilinearity follows from the linearity of the local residue and of the trace.
\end{proof}

\begin{lemma}[Perfectness of the general finite residue pairing]
\label{lemma:serre-duality-D}
The finite residue pairing \(\langle \cdot,\cdot\rangle_D\) is perfect.
Locally at a point \(x\in\supp D\) with multiplicity \(m=\operatorname{ord}_x(D)\),
choose a uniformizer \(u\) and a local trivialization \(e\) of \(\mathcal M\).
Let \(\{\alpha_1,\dots,\alpha_s\}\) be an \(\F_q\)-basis of \(\kappa(x)\), and
let \(\{\alpha_1^*,\dots,\alpha_s^*\}\) be the dual basis with respect to the
trace pairing
\[
\Tr_{\kappa(x)/\F_q}(\alpha_k\alpha_l^*)=\delta_{kl}.
\]
Then the families
\[
\{\alpha_k u^i e\}_{0\le i<m,\,1\le k\le s}
\qquad\text{and}\qquad
\left\{\alpha_l^*\,\frac{du}{u^{j+1}}\,e^{-1}\right\}_{0\le j<m,\,1\le l\le s}
\]
are dual bases for the local factors of
\(H^0(D,\mathcal{M}|_D)\) and
\(H^0\left(D,\frac{(\mathcal{M}^{-1}\otimes\OmegaOne)(D)}
{\mathcal{M}^{-1}\otimes\OmegaOne}\right)\), respectively.
\end{lemma}

\begin{proof}
The pairing decomposes as a direct sum over the points of \(\supp D\).
Thus it suffices to prove perfectness locally at one point \(x\).

Choose a uniformizer \(u\) and a local trivialization \(e\) of \(\mathcal M\)
near \(x\). Then
\[
\mathcal M|_D
\cong
\left(\kappa(x)[u]/(u^m)\right)\cdot e,\]
and
\[
\frac{(\mathcal{M}^{-1}\otimes\OmegaOne)(D)}
{\mathcal{M}^{-1}\otimes\OmegaOne}
\cong
\left(\kappa(x)[u]du/(u^m du)\right)\cdot e^{-1}.
\]
Let \(\{\alpha_k\}\) and \(\{\alpha_l^*\}\) be trace-dual bases as above.
Define basis elements
\[
e_{k,i}=\alpha_k u^i e,
\qquad
e_{l,j}^{\dagger}=\alpha_l^*\,\frac{du}{u^{j+1}}\,e^{-1}.
\]
A direct residue computation gives
\[
\langle e_{k,i},e_{l,j}^{\dagger}\rangle_D
=
\delta_{ij}\,
\Tr_{\kappa(x)/\F_q}(\alpha_k\alpha_l^*)
=
\delta_{ij}\delta_{kl}.
\]
Thus the two bases are dual, and the local pairing is perfect.
Summing over all \(x\in\supp D\) proves the global statement.
\end{proof}

Base change to \(\Cinf\) yields the same statement without traces:
if \(\mathcal M\) is a line bundle on \(U_{\Cinf}\) and \(D\) is an effective
divisor on \(U_{\Cinf}\), then the pairing
\begin{align*}
H^0(D,\mathcal M|_D)
\times
H^0\left(D,\frac{(\mathcal M^{-1}\otimes\OmegaOne)(D)}
{\mathcal M^{-1}\otimes\OmegaOne}\right)
&\longrightarrow \Cinf,\\
(m,\xi) &\mapsto \sum_{y\in\supp D}\res_y(m\cdot\tilde\xi),
\end{align*}
is perfect. At a \(\Cinf\)-point \(y\), with uniformizer \(u\) and local
trivialization \(e\), the bases
\[
\{u^k e\}_{0\le k<m}
\qquad\text{and}\qquad
\left\{\frac{du}{u^{k+1}}e^{-1}\right\}_{0\le k<m}
\]
are dual. This base-changed pairing is the scalar extension of the
\(\F_q\)-pairing defined above.

\begin{notation}
For line bundle $\LL$, we define the vector space
\[
    W_D(\LL) := H^0(D, \LL \mid_D) \otimes_{\Cinf} H^0\left(D, \frac{\LL^{-1}\otimes \OmegaOne(D)}{\LL^{-1}\otimes \OmegaOne}\right).
    \]
\end{notation}
\begin{proposition}[Endomorphism interpretation of the principal part space]
\label{prop:endomorphism-interpretation}
Let \(\mathcal{L}\) be a line bundle on \(U_{\Cinf}\), and let \(D\) be a
finite effective divisor on \(U_{\Cinf}\). The perfect residue pairing
\(\langle \cdot,\cdot\rangle_D\) induces a natural isomorphism
\[
\Phi_D : W_D(\mathcal{L})
\stackrel{\sim}{\longrightarrow}
\operatorname{End}_{\Cinf}\bigl(H^0(D,\mathcal{L}|_D)\bigr).
\]
On pure tensors it is given by
\[
\Phi_D(m\otimes \xi)(v)
=
m \cdot \langle v,\xi\rangle_D,
\]
for \(m,v \in H^0(D,\mathcal{L}|_D)\) and
\(\xi \in H^0(D,(\mathcal{L}^{-1}\otimes\Omega^1)(D)/(\mathcal{L}^{-1}\otimes\Omega^1))\).
\end{proposition}

\begin{proof}
Put \(V=H^0(D,\mathcal{L}|_D)\) and
\(W=H^0(D,(\mathcal{L}^{-1}\otimes\Omega^1)(D)/(\mathcal{L}^{-1}\otimes\Omega^1))\).
The Serre duality pairing
\[
\langle \cdot,\cdot\rangle_D : V\times W \to \Cinf
\]
is perfect. Hence the map
\[
\Psi : W \to V^*,
\qquad
\Psi(\xi)(v)=\langle v,\xi\rangle_D,
\]
is an isomorphism. Tensoring with the identity on \(V\), we obtain
\[
\id_V \otimes \Psi : V\otimes W \xrightarrow{\sim} V\otimes V^*.
\]
The canonical isomorphism \(V\otimes V^* \cong \operatorname{End}(V)\) sends
\(v\otimes \lambda\) to the rank-one endomorphism \(u\mapsto v\cdot \lambda(u)\).
Thus the composition
\[
\Phi_D = (\id_V\otimes \Psi)\circ \iota_{V,V^*}
\]
is an isomorphism, where \(\iota_{V,V^*}:V\otimes V^*\to \operatorname{End}(V)\)
is the natural evaluation map. Unwinding the definitions gives
\[
\Phi_D(m\otimes \xi)(v)=m\cdot \langle v,\xi\rangle_D,
\]
as claimed.

Conversely, if \(\{v_i\}\) is a basis of \(V\) and \(\{\xi_i\}\) is the dual
basis of \(W\) with respect to the residue pairing, i.e.
\(\langle v_i,\xi_j\rangle_D=\delta_{ij}\), then for any \(T\in\operatorname{End}(V)\),
\[
\Phi_D^{-1}(T)=\sum_i T(v_i)\otimes \xi_i.
\]
This formula is independent of the chosen basis and exhibits the naturality
of the isomorphism.
\end{proof}


\subsection{The Finite Residue Pairing on Dedekind Domain}
\label{subsec:finite-residue-pairing}

Let \(\f\subseteq A\) be a nonzero ideal, with associated effective divisor
\(D_\f\) on \(U\). We now specialize the general Serre duality of
Section~\ref{subsec:serre-duality-divisor} to the case
\[
\mathcal M=\mathcal O_U.
\]

\begin{definition}
Define the \textbf{function quotient module}
\[
V_{\f}=A/\f.
\]
Restriction of functions to the divisor \(D_\f\) induces a canonical
isomorphism
\[
V_{\f}\cong H^0(D_\f,\mathcal O_U|_{D_\f})
=
H^0(D_\f,\mathcal O_{D_\f}).
\]
This is the first space in the general Serre pairing with
\(\mathcal M=\mathcal O_U\).
\end{definition}

Dually, define the differential quotient module as the second space in the
same pairing.

\begin{definition}
Define the \textbf{differential quotient module}
\[
V_{\f}^{\dagger}
=
\frac{\f^{-1}\OmegaOne_A}{\OmegaOne_A}.
\]
Restriction to the divisor induces a canonical isomorphism
\[
V_{\f}^{\dagger}
\cong
H^0\left(D_\f,\frac{\OmegaOne_U(D_\f)}{\OmegaOne_U}\right).
\]
This is the second space in the general Serre pairing with
\(\mathcal M=\mathcal O_U\).
\end{definition}

Both \(V_{\f}\) and \(V_{\f}^{\dagger}\) are finite-dimensional \(\F_q\)-vector
spaces of dimension \(\deg D_\f\). They are precisely the special cases
\[
H^0(D_\f,\mathcal M|_{D_\f})
\quad\text{and}\quad
H^0\left(D_\f,\frac{(\mathcal M^{-1}\otimes\OmegaOne)(D_\f)}
{\mathcal M^{-1}\otimes\OmegaOne}\right)
\]
when \(\mathcal M=\mathcal O_U\).

\begin{remark}
For \(A=\F_q[t]\), the differential module \(\OmegaOne_A\) is free of rank one,
generated by \(dt\). The map \(g\mapsto g\,dt/f\) gives an isomorphism
\(V_{\f}\xrightarrow{\sim}V_{\f}^{\dagger}\). Over a general ring \(A\),
however, \(\OmegaOne_A\) need not be free, so no such canonical isomorphism
exists. This is why the general Serre picture works with an arbitrary line
bundle \(\mathcal M\), while the algebraic pairing specializes to
\(\mathcal M=\mathcal O_U\).
\end{remark}



\begin{definition}
The \textbf{finite residue pairing} is the map
\[
\langle \cdot,\cdot\rangle_{\f}: V_{\f}\times V_{\f}^{\dagger}\to \F_q
\]
given by
\[
\langle \overline a,\overline\omega\rangle_{\f}
=
\Res_{D_\f}(a\omega)
=
\sum_{x\in\supp D_\f}
\Tr_{\kappa(x)/\F_q}\left(\res_x(a\omega)\right),
\]
where \(\overline a\in V_{\f}\) is represented by \(a\in A\), and
\(\overline\omega\in V_{\f}^{\dagger}\) is represented by
\(\omega\in\f^{-1}\OmegaOne_A\).
\end{definition}

Well-definedness and bilinearity are immediate from the general case in
Section~\ref{subsec:serre-duality-divisor}. Perfectness is also a special
case:

\begin{theorem}[Perfectness of the finite residue pairing]
\label{thm:perfect-pairing}
The finite residue pairing \(\langle \cdot,\cdot\rangle_{\f}\) is perfect.
Equivalently, the induced maps
\[
V_{\f}\to \Hom_{\F_q}(V_{\f}^{\dagger},\F_q),
\qquad
V_{\f}^{\dagger}\to \Hom_{\F_q}(V_{\f},\F_q)
\]
are isomorphisms.
\end{theorem}

\begin{proof}
This is the special case \(\mathcal M=\mathcal O_U\) of
Lemma~\ref{lemma:serre-duality-D}. The explicit dual bases given there become
the local dual bases used in the earlier proof of perfectness.
\end{proof}

\subsection{The Finite Residue Casimir}
\label{subsec:finite-casimir}

A perfect pairing identifies each vector space with the dual of the other. Associated to any perfect pairing is a canonical tensor element that represents the identity endomorphism under this identification: the \textbf{Casimir tensor}.

Let $\{ \alpha_i \}$ be any basis of $V_{\f}$, and let $\{ \alpha_i^{\dagger} \}$ be the unique dual basis of $V_{\f}^{\dagger}$ satisfying
\[
\pair{\alpha_i}{\alpha_j^{\dagger}}_{\f} = \delta_{ij}.
\]

\begin{definition}[Finite residue Casimir]
The \textbf{finite residue Casimir} is the element
\begin{equation}
\label{eq:casimir-def}
\CasFin = \sum_i \alpha_i \otimes \alpha_i^{\dagger} \quad \in V_{\f} \otimes_{\F_q} V_{\f}^{\dagger}.
\end{equation}
\end{definition}

A priori this definition depends on the choice of basis. The following lemma shows that it does not.

\begin{lemma}[Basis independence]\label{lem:basis_independence}
The Casimir tensor $\CasFin$ defined in \eqref{eq:casimir-def} is independent of the choice of basis $\{\alpha_i\}$. It is characterized by either of the following two equivalent properties:
\begin{enumerate}
\item[(i)] $\displaystyle \sum_i \alpha_i \pair{v}{\alpha_i^{\dagger}}_{\f} = v$ \quad for all $v \in V_{\f}$;
\item[(ii)] $\displaystyle \sum_i \pair{\alpha_i}{\omega}_{\f} \alpha_i^{\dagger} = \omega$ \quad for all $\omega \in V_{\f}^{\dagger}$.
\end{enumerate}
\end{lemma}

\begin{proof}
We first verify that properties (i) and (ii) hold for any dual basis pair. For property (i), write $v = \sum_j c_j \alpha_j$. Then
\[
\sum_i \alpha_i \pair{v}{\alpha_i^{\dagger}}_{\f} = \sum_i \alpha_i \sum_j c_j \pair{\alpha_j}{\alpha_i^{\dagger}}_{\f} = \sum_i \alpha_i \sum_j c_j \delta_{ji} = \sum_i c_i \alpha_i = v.
\]
Property (ii) is proved similarly.

Now, any element satisfying property (i) is uniquely determined, because the map
\[
V_{\f} \otimes V_{\f}^{\dagger} \to \End(V_{\f}), \qquad a \otimes \omega \mapsto (v \mapsto a \cdot \pair{v}{\omega})
\]
is an isomorphism (by finite dimensionality and perfectness of the pairing). Under this isomorphism, the Casimir corresponds to the identity map $\id_{V_{\f}}$, which is certainly basis-independent.
\end{proof}

Under the canonical isomorphism $V_{\f} \otimes V_{\f}^{\dagger} \cong \End(V_{\f})$ (see Proposition~\ref{prop:endomorphism-interpretation}), the Casimir tensor \eqref{eq:casimir-def} is simply the identity operator. This is the most important fact to remember about it.

\subsection{Scalar Extension and the Arithmetic Bridge}
\label{subsec:scalar-extension}

We now relate the algebraic Casimir over $\F_q$ to its geometric counterpart over $\Cinf$.

Base change along the field extension $\F_q \hookrightarrow \Cinf$ identifies
\[
H^0(D, \OO) \cong V_{\f} \otimes_{\F_q} \Cinf, \qquad H^0\left(D, \frac{\OmegaOne(D)}{\OmegaOne}\right) \cong V_{\f}^{\dagger} \otimes_{\F_q} \Cinf.
\]
This defines a canonical scalar extension map
\begin{equation}
\label{eq:inc-def}
\inc : V_{\f} \otimes_{\F_q} V_{\f}^{\dagger} \longrightarrow W_D(\OO) := H^0(D, \OO) \otimes_{\Cinf} H^0\left(D, \frac{\OmegaOne(D)}{\OmegaOne}\right).
\end{equation}

The following lemma gives an explicit local formula for $\inc(\CasFin)$ in terms of uniformizers. It is the tensor form of the base-change lemma for residues (Lemma \ref{lemma:base-change-residues}), and is the only arithmetic input to our main theorem.

\begin{lemma}[Split form of the Casimir]
\label{lemma:split-casimir}
For any choice of uniformizers $u_y$ at the points $y \in \supp{D}$, we have
\begin{equation}
\label{eq:split-casimir}
\inc\left(\CasFin\right) = \sum_{y \in \supp{D}} \sum_{k=0}^{m_y - 1} u_\one^k \otimes \frac{du_{\two}}{u_{\two}^{k+1}}.
\end{equation}
In particular, the right-hand side is independent of the choice of uniformizers.
\end{lemma}

\begin{proof}
Take a point $x \in \supp{\Df}$ of degree $s$, with uniformizer $v$ and trace-dual bases $\{\beta_l\}$, $\{\beta_l^*\}$ of $\kappa(x) / \F_q$. From the proof of Theorem \ref{thm:perfect-pairing}, the local contribution to the Casimir at $x$ is
\[
\sum_{l=1}^s \sum_{k=0}^{m_x - 1} (\beta_l v^k) \otimes \left( \beta_l^* \frac{dv}{v^{k+1}} \right).
\]

Under base change, each such tensor maps to
\[
\sum_{l=1}^s \sum_{k=0}^{m_x - 1} \sum_{j', j = 1}^s \sigma_{j'}(\beta_l) \sigma_j(\beta_l^*) \, u_{j'}^k \otimes \frac{du_j}{u_j^{k+1}},
\]
where $u_j$ is the uniformizer at the split point $y_j$ (which is just $v$ regarded as an element of $\Cinf[[v]]$).

We now use the matrix identity
\[
\sum_{l=1}^s \sigma_{j'}(\beta_l) \sigma_j(\beta_l^*) = \delta_{j'j}.
\]
This identity follows from the definition of dual bases: for any $\gamma \in \kappa(x)$, we have
\[
\sum_j \sigma_j(\gamma \beta_l^*) = \Tr(\gamma \beta_l^*),
\]
and summing over $l$ gives
\[
\sum_l \sum_j \sigma_j(\beta_l) \sigma_j(\beta_l^*) = \sum_j \sigma_j \left( \sum_l \beta_l \beta_l^* \right) = \sum_j 1 = s,
\]
but more precisely, for fixed $j', j$, the sum over $l$ is the $(j', j)$ entry of the product of the basis matrix and its dual, which is the identity matrix. Indeed, the map sending $\gamma$ to $(\sigma_1(\gamma), \dots, \sigma_s(\gamma))$ embeds $\kappa(x)$ into $\Cinf^s$, and the trace-dual bases map to standard dual bases under this embedding. Hence $\sum_l \sigma_{j'}(\beta_l) \sigma_j(\beta_l^*) = \delta_{j'j}$.

Applying this identity collapses the double sum over $j', j$ to a single sum:
\[
\sum_{j=1}^s \sum_{k=0}^{m_x - 1} u_j^k \otimes \frac{du_j}{u_j^{k+1}}.
\]
Summing over all $x \in \supp{\Df}$ yields the stated formula \eqref{eq:split-casimir}.

The independence of uniformizers follows from the fact that both sides represent the tensor of the identity map for the base-changed pairing. Lemma \ref{lemma:serre-duality-D} shows that for any uniformizer $u$, the bases $\{u^k\}$ and $\{du/u^{k+1}\}$ are dual, so the sum on the right always represents the identity endomorphism. Since the map 
\[
W_D(\OO) \to \End(H^0(D, \OO))
\]
is injective, the element is uniquely determined, hence independent of the choice of uniformizers.
\end{proof}

\begin{remark}
Lemma \ref{lemma:split-casimir} is the tensor incarnation of Lemma \ref{lemma:base-change-residues}. It says that the Casimir, which is defined arithmetically using traces over finite fields, becomes a simple sum of local ``identity tensors'' after base change to the algebraically closed field $\Cinf$. This is the key fact that allows us to compare the algebraic Casimir with the geometric principal part of the Szeg\H{o} kernel.
\end{remark}

\section{Acyclic Bundles and the Szeg\H{o} Kernel}
\label{sec:acyclic-szego}

In this section we move from algebra to geometry. We introduce acyclic line bundles on $X_{\Cinf}$, define the Szeg\H{o} bundle on the product curve, and prove the existence and uniqueness of the normalized Szeg\H{o} kernel. We also derive its transformation law under twisting by a line bundle that is trivial near our finite divisor.

\subsection{Acyclic Line Bundles}
\label{subsec:acyclic-bundles}

\begin{definition}
A line bundle $\LL$ on $X_{\Cinf}$ is called \textbf{acyclic} if
\[
H^0(X_{\Cinf}, \LL) = 0 \quad \text{and} \quad H^1(X_{\Cinf}, \LL) = 0.
\]
\end{definition}

By the Riemann--Roch theorem,
\[
h^0(\LL) - h^1(\LL) = \deg \LL + 1 - g.
\]
If $\LL$ is acyclic, the left-hand side is zero, so $\deg \LL = g - 1$. Thus acyclic line bundles are precisely the degree $g-1$ line bundles that have no nonzero global sections.

Note that $\LL$ is acyclic if and only if its Serre dual $\LL^{-1} \otimes \OmegaOne$ is acyclic. This follows immediately from Serre duality, which gives isomorphisms
\[
H^0(\LL)^{\vee} \cong H^1(\LL^{-1} \otimes \OmegaOne), \qquad H^1(\LL)^{\vee} \cong H^0(\LL^{-1} \otimes \OmegaOne).
\]

We now prove that acyclic line bundles always exist, in every genus.

\begin{proposition}[Existence of acyclic bundles]
Acyclic line bundles exist on $X_{\Cinf}$ for every genus $g \geq 0$. More precisely, the set of acyclic line bundles is
\[
\{ \LL \text{ acyclic} \} = \operatorname{Pic}^{g-1}(X_{\Cinf}) \setminus (\Theta \cup (K - \Theta)),
\]
where $\Theta = \{ \LL : h^0(\LL) \geq 1 \}$ is the theta divisor and $K - \Theta$ is its Serre dual translate. This is a non-empty Zariski-open subset of the $g$-dimensional Picard variety.
\end{proposition}

\begin{proof}
By definition, $\LL$ is acyclic iff $h^0(\LL) = 0$ and $h^1(\LL) = 0$. By Serre duality, $h^1(\LL) = h^0(\LL^{-1} \otimes \OmegaOne)$. So acyclicity is equivalent to
\[
\LL \notin \Theta \quad \text{and} \quad \LL^{-1} \otimes \OmegaOne \notin \Theta.
\]
The second condition is equivalent to $\LL \notin K - \Theta$, where $K$ is the canonical divisor class.

Now, $\Theta$ is the image of the Abel--Jacobi map from the $(g-1)$-th symmetric power $X^{(g-1)}$ to $\operatorname{Pic}^{g-1}$. For $g \geq 2$, the dimension of $X^{(g-1)}$ is $g-1$, which is strictly less than $g$, so $\Theta$ is a proper closed subset of the Picard variety. For $g = 1$, $\Theta$ is a single point (the zero divisor class), still proper. For $g = 0$, $\Theta$ is empty because there are no effective divisors of degree $-1$.

Similarly, $K - \Theta$ is also a proper closed subset, being a translate of $\Theta$.

Over an algebraically closed field of infinite cardinality (such as $\Cinf$), a variety of positive dimension cannot be covered by finitely many proper closed subsets. Therefore the complement of $\Theta \cup (K - \Theta)$ is non-empty and Zariski open.
\end{proof}
\begin{example}\label{ex:szego}
For $g = 0$, we can give an explicit example. Take $\LL = \OO(-[P])$ for any $\F_q$-rational point $P$. This bundle has degree $-1 = g-1$, and clearly $h^0 = 0$ since it has negative degree. Then $h^1 = 0$ by Riemann--Roch.
\end{example}
\subsection{The Szeg\H{o} Bundle and the Normalized Szeg\H{o} Kernel}
\label{subsec:szego-bundle-kernel}

We now work on the product curve $X_{\Cinf} \times X_{\Cinf}$, with projections
\[
\prI, \prII : X_{\Cinf} \times X_{\Cinf} \longrightarrow X_{\Cinf}
\]
onto the first and second factors respectively. Let $\DeltaDiag \subset X_{\Cinf} \times X_{\Cinf}$ denote the diagonal subscheme.

\begin{definition}
For any line bundle $\LL$ on $X_{\Cinf}$, define the \textbf{Szeg\H{o} bundle}
\begin{equation}
\label{eq:szego-bundle-def}
\BB_{\LL} = \prI^* \LL \otimes \prII^* \left( \LL^{-1} \otimes \OmegaOne \right).
\end{equation}
We also consider the twisted bundle
\begin{equation}
\label{eq:szego-bundle-twisted-def}
\BB_{\LL}(\DeltaDiag) = \BB_{\LL} \otimes \OO(\DeltaDiag),
\end{equation}
which allows sections with a simple pole along the diagonal.
\end{definition}

\begin{lemma}\label{lem:bundle}
The restriction of $\BB_{\LL}(\DeltaDiag)$ to the diagonal $\DeltaDiag$ is canonically trivial. 
\end{lemma}
\begin{proof}
Note that the restriction of $\BB_{\LL}$ to $\DeltaDiag$ is $\LL \otimes (\LL^{-1} \otimes \OmegaOne) \cong \OmegaOne$.
The restriction of $\OO(\DeltaDiag)$ to $\DeltaDiag$ is the normal bundle $N_{\DeltaDiag}$, which is isomorphic to the tangent bundle $T \cong (\OmegaOne)^{-1}$.
Therefore, their tensor product is $\OmegaOne \otimes (\OmegaOne)^{-1} \cong \OO_{\DeltaDiag}$.
\end{proof}

Lemma \ref{lem:bundle} gives rise to the diagonal residue exact sequence:
\begin{equation}
\label{eq:diagonal-residue-sequence}
0 \longrightarrow \BB_{\LL} \longrightarrow \BB_{\LL}(\DeltaDiag) \stackrel{\res_{\DeltaDiag}}{\longrightarrow} \OO_{\DeltaDiag} \longrightarrow 0.
\end{equation}
Locally, near a point $(y,y)$ on the diagonal, if $u$ is a local coordinate and $e$ is a local frame for $\LL$, then a section of the form
\[
\frac{\sigma \, du_{\two}}{u_{\two} - u_\one} \, e(\one) \otimes e(\two)^{-1}
\]
has diagonal residue equal to $\sigma|_{\DeltaDiag}$.

We now come to the central geometric object of this paper.

\begin{proposition}[Existence and uniqueness of the Szeg\H{o} kernel]
\label{prop:szego-existence}
Let $\LL$ be an acyclic line bundle on $X_{\Cinf}$. There exists a unique global section
\[
S_{\LL} \in H^0(X_{\Cinf} \times X_{\Cinf}, \BB_{\LL}(\DeltaDiag))
\]
such that $\res_{\DeltaDiag}(S_{\LL}) = 1$. This section is called the \textbf{normalized Szeg\H{o} kernel} of $\LL$.

Locally near any diagonal point $(y,y)$, in any uniformizer $u$ and any local frame $e$ of $\LL$ near $y$, the Szeg\H{o} kernel admits the expansion
\begin{equation}
\label{eq:szego-local-expansion}
S_{\LL} = \frac{du_{\two}}{u_{\two} - u_\one} \, e(\one) \otimes e(\two)^{-1} + \epsilon,
\end{equation}
where $\epsilon$ is a holomorphic section of $\BB_{\LL}$ (i.e., regular along the diagonal).
\end{proposition}

\begin{proof}
Consider the long exact cohomology sequence associated to the diagonal residue short exact sequence \eqref{eq:diagonal-residue-sequence}:
\[
\cdots \to H^0(\BB_{\LL}) \to H^0(\BB_{\LL}(\DeltaDiag)) \stackrel{\res_{\DeltaDiag}}{\longrightarrow} H^0(\OO_{\DeltaDiag}) \to H^1(\BB_{\LL}) \to \cdots.
\]

We claim that $H^0(\BB_{\LL}) = 0$ and $H^1(\BB_{\LL}) = 0$. By the K\"unneth formula,
\[
H^i(\BB_{\LL}) = \bigoplus_{p+q = i} H^p(\LL) \otimes H^q(\LL^{-1} \otimes \OmegaOne).
\]
Since $\LL$ is acyclic, $H^0(\LL) = H^1(\LL) = 0$. By Serre duality, $H^0(\LL^{-1} \otimes \OmegaOne) = H^1(\LL)^{\vee} = 0$ and $H^1(\LL^{-1} \otimes \OmegaOne) = H^0(\LL)^{\vee} = 0$. Therefore all K\"unneth summands vanish for $i = 0, 1$.

Thus the long exact sequence gives an isomorphism
\[
H^0(\BB_{\LL}(\DeltaDiag)) \stackrel{\res_{\DeltaDiag}}{\xrightarrow{\sim}} H^0(\OO_{\DeltaDiag}) = \Cinf.
\]
Hence there is a unique section of $\BB_{\LL}(\DeltaDiag)$ whose diagonal residue equals $1 \in \Cinf$. This is the normalized Szeg\H{o} kernel.

For the local expansion, write
\[
S_{\LL} = \frac{f \, du_{\two}}{u_{\two} - u_\one} \, e(\one) \otimes e(\two)^{-1} + \text{regular terms},
\]
where $f$ is a holomorphic function. By definition, the diagonal residue is $f|_{\DeltaDiag}$. Since the residue equals $1$, we have $f|_{\DeltaDiag} = 1$. We can absorb $f - 1$ into the regular term $\epsilon$, which yields the stated local form \eqref{eq:szego-local-expansion}.
\end{proof}

\subsection{Twisting Property of Szeg\H{o} Kernels}
\label{subsec:twisting-property}

The Szeg\H{o} kernel transforms in a simple way when the line bundle is twisted by another line bundle that is trivial on some open set. This property will be crucial in the proof of the main theorem, where it will allow us to reduce the computation to a single choice of line bundle.

\begin{proposition}[Twisting formula]
\label{prop:twisting-szego}
Let $\LL_1$ and $\LL_2$ be acyclic line bundles, and let $\MM = \LL_2 \otimes \LL_1^{-1}$. Suppose $\MM$ is trivial on an open subset $V \subseteq X_{\Cinf}$, and let $m$ be a nowhere-vanishing section of $\MM$ on $V$. Then on $V \times V$,
\begin{equation}
\label{eq:twisting-szego}
S_{\LL_2} = \frac{m(\one)}{m(\two)} \, S_{\LL_1}.
\end{equation}
More generally, if $m$ trivializes $\MM$ only on an open neighbourhood $V \supseteq \supp{D}$, then the identity holds on $V \times V$ up to a section of $\BB_{\LL_2}$ that is holomorphic there.
\end{proposition}

\begin{proof}
The right-hand side $\frac{m(\one)}{m(\two)} S_{\LL_1}$ is a section of
\[
\prI^* \MM \otimes \prII^* \MM^{-1} \otimes \BB_{\LL_1}(\DeltaDiag) = \BB_{\LL_2}(\DeltaDiag)
\]
on $V \times V$. Its diagonal residue is
\[
\res_{\DeltaDiag}\left( \frac{m(\one)}{m(\two)} S_{\LL_1} \right) = \frac{m}{m} \cdot \res_{\DeltaDiag}(S_{\LL_1}) = 1 \cdot 1 = 1.
\]

Thus both $S_{\LL_2}$ and $\frac{m(\one)}{m(\two)} S_{\LL_1}$ are sections of $\BB_{\LL_2}(\DeltaDiag)$ on $V \times V$ with diagonal residue 1. Their difference is a section of $\BB_{\LL_2}$ (holomorphic, no pole along the diagonal).

In the global case (where $V = X_{\Cinf}$), we have $H^0(\BB_{\LL_2}) = 0$ by acyclicity, so the difference must be zero, and the identity \eqref{eq:twisting-szego} holds exactly.

In the local case (where $V$ is only a neighbourhood of $\supp{D}$), we do not have vanishing of global sections, but we only need the fact that the difference is holomorphic on $V \times V$. This will be sufficient for our applications to principal parts, since holomorphic sections have zero principal part.
\end{proof}

\section{The Intrinsic Principal Part}
\label{sec:intrinsic-pp}

This section contains the technical heart of the paper: a coordinate-free definition of the principal part of a section along a finite divisor $D$. Previous treatments defined principal parts by truncating Laurent series in local coordinates, which then required separate lemmas to prove independence of coordinates and frames. Our intrinsic definition, by contrast, realizes the principal part as the class of a restriction in a natural quotient space. All invariance properties then follow formally from functoriality.

\subsection{Definition of the Intrinsic Principal Part}
\label{subsec:def-intrinsic-pp}

Let $\LL$ be a line bundle on $U_{\Cinf}$, and let $U \supseteq \supp{D}$ be an affine open neighbourhood of the support of $D$. Consider the product $D \times U$, which is a finite scheme over $U$ via the second projection.

On $D \times U$, we have two natural sheaves:
\begin{itemize}
\item $\prI^* (\LL|_D)$, the pullback of the restricted line bundle from the first factor;
\item $\prII^* ((\LL^{-1} \otimes \OmegaOne)(D))$, the pullback of the twisted differential sheaf from the second factor, which allows poles of order up to the multiplicity of $D$.
\end{itemize}

There is a natural subsheaf
\[
\prII^* (\LL^{-1} \otimes \OmegaOne) \subseteq \prII^* ((\LL^{-1} \otimes \OmegaOne)(D)),
\]
consisting of sections with no poles along $D$.

\begin{definition}[Intrinsic principal part space]
Define the \textbf{principal part space}
\begin{equation}
\label{eq:intrinsic-pp-space}
\overline{W}_D(\LL) = \frac{H^0\left(D \times U, \, \prI^* (\LL|_D) \otimes \prII^* ((\LL^{-1} \otimes \OmegaOne)(D))\right)}{H^0\left(D \times U, \, \prI^* (\LL|_D) \otimes \prII^* (\LL^{-1} \otimes \OmegaOne)\right)}.
\end{equation}
\end{definition}

Now let $s \in H^0(U \times U, \BB_{\LL}(\DeltaDiag))$ be a section defined on $U \times U$. We can restrict $s$ to the closed subscheme $D \times U \subseteq U \times U$. Since $s$ has at most a simple pole along the diagonal, and $D$ is a finite divisor, the restriction $s|_{D \times U}$ defines a section of the numerator sheaf above. Indeed, the simple pole along the diagonal becomes a pole of order at most $m_y$ in the second variable when restricted to $D \times \{y\}$.

\begin{definition}[Intrinsic principal part]
\label{def:intrinsic-pp}
For $s \in H^0(U \times U, \BB_{\LL}(\DeltaDiag))$, define its \textbf{intrinsic principal part along $D$} by
\begin{equation}
\label{eq:intrinsic-pp-def}
\PP_D(s) = \left[ s|_{D \times U} \right] \in \overline{W}_D(\LL),
\end{equation}
the class of the restriction in the quotient space.
\end{definition}

When no confusion is likely to arise, we will identify $\overline{W}_D(\LL)$ with the ``model'' space
\[
W_D(\LL) = H^0(D, \LL|_D) \otimes_{\Cinf} H^0\left(D, \frac{(\LL^{-1} \otimes \OmegaOne)(D)}{\LL^{-1} \otimes \OmegaOne}\right)
\]
via the canonical K\"unneth isomorphism (proved in Lemma \ref{lemma:pp-basic-properties} below). This identifies our intrinsic principal part with the more familiar tensor-product form.

\subsection{Basic Properties of the Principal Part}
\label{subsec:basic-pp-properties}

We now derive the fundamental properties of the intrinsic principal part. All of these follow directly from the definition and basic algebraic geometry.

\begin{lemma}[Basic properties of $\PP_D$]
\label{lemma:pp-basic-properties}
Let $s \in H^0(U \times U, \BB_{\LL}(\DeltaDiag))$.
\begin{enumerate}
\item[(i)] \textbf{Well-definedness and exact truncation.} The restriction $s|_{D \times U}$ is indeed a section of the numerator sheaf. Locally at $y \in \supp{D}$ with multiplicity $m = m_y$ and uniformizer $u$, the identity
\begin{equation}
\label{eq:exact-truncation}
\left. \frac{1}{u_{\two} - u_\one} \right|_{D \times U} = \frac{\sum_{j < m} u_{\two}^{m-1-j} u_\one^j}{u_{\two}^m} = \sum_{j=0}^{m-1} \frac{u_\one^j}{u_{\two}^{j+1}}
\end{equation}
holds exactly on $D \times U$, with no convergence issues.

\item[(ii)] \textbf{Identification with the model space.} Since $D$ is finite and $U$ is affine, there is a canonical isomorphism
\[
\overline{W}_D(\LL) \cong W_D(\LL),
\]
compatible with shrinking $U$.

\item[(iii)] \textbf{Regularity.} If $s$ is a section of the denominator sheaf (i.e., holomorphic in the second variable along $D$), then $\PP_D(s) = 0$.

\item[(iv)] \textbf{Functoriality.} An isomorphism $\varphi : \LL_1 \xrightarrow{\sim} \LL_2$ on a neighbourhood of $\supp{D}$ induces an isomorphism $\varphi_* : \overline{W}_D(\LL_1) \cong \overline{W}_D(\LL_2)$ such that
\[
\varphi_*(\PP_D(s)) = \PP_D(\varphi(s)).
\]
When $\varphi$ is a frame trivialization $e : \OO_U \xrightarrow{\sim} \LL|_U$, this recovers the compatibility between principal parts and $e$-coefficients:
\[
\UpsilonFrame(\PP_D(s)) = \PP_D(\Upsilon_e^U(s)).
\]

\item[(v)] \textbf{Local shape.} For the normalized Szeg\H{o} kernel $S_{\LL}$, in any frame $e$ and any choice of uniformizers $u_y$ at points $y \in \supp{D}$, we have
\begin{equation}
\label{eq:pp-local-shape}
\PP_D(S_{\LL}^e) = \sum_{y \in \supp{D}} \sum_{k=0}^{m_y - 1} u_\one^k \otimes \frac{du_{\two}}{u_{\two}^{k+1}}.
\end{equation}
\end{enumerate}
\end{lemma}

\begin{proof}
(i) Near a diagonal point $(y,y)$, write $s = \sigma \, du_{\two} / (u_{\two} - u_\one)$ with $\sigma$ holomorphic. We need to show that $1/(u_{\two} - u_\one)$, when restricted to $D \times U$ (where $u_\one^m = 0$), has a pole of order at most $m$ in $u_{\two}$.

The algebraic identity
\[
X^m - Y^m = (X - Y) \sum_{j=0}^{m-1} X^j Y^{m-1-j}
\]
implies
\[
\frac{1}{Y - X} = \frac{\sum_{j=0}^{m-1} X^j Y^{m-1-j}}{Y^m - X^m}.
\]
On $D \times U$, we have $X^m = u_\one^m = 0$, so this simplifies to
\[
\frac{1}{u_{\two} - u_\one} = \frac{\sum_{j=0}^{m-1} u_\one^j u_{\two}^{m-1-j}}{u_{\two}^m} = \sum_{j=0}^{m-1} \frac{u_\one^j}{u_{\two}^{j+1}},
\]
which is \eqref{eq:exact-truncation}. This is an exact algebraic identity, valid in any characteristic, with no convergence issues. Away from diagonal points, $s|_{D \times U}$ is holomorphic in the second variable, so it contributes nothing to the principal part.

(ii) By K\"unneth, since $D$ is finite and $U$ is affine, we have
\[
H^0(D \times U, \prI^* \FF \otimes \prII^* \GG) \cong H^0(D, \FF) \otimes H^0(U, \GG)
\]
for any coherent sheaves $\FF$ on $D$ and $\GG$ on $U$. Applying this to both numerator and denominator, and using the fact that the restriction map $H^0(U, \GG(D)) \to H^0(D, \GG(D)/\GG)$ is surjective (by affineness), we obtain the canonical isomorphism.

(iii) Immediate from the definition: sections of the denominator represent the zero class in the quotient.

(iv) The isomorphism $\varphi$ preserves pole orders on both factors, so it maps numerator to numerator and denominator to denominator, hence induces a map on the quotient. It commutes with restriction to $D \times U$, so it intertwines the principal part maps. The frame case is just the special case where $\LL_1 = \OO$ and $\LL_2 = \LL$.

(v) From Proposition \ref{prop:szego-existence}, the Szeg\H{o} kernel has the local form
\[
S_{\LL}^e = \frac{du_{\two}}{u_{\two} - u_\one} + \epsilon,
\]
where $\epsilon$ is holomorphic (regular in both variables). By part (iii), the regular term $\epsilon$ has zero principal part. Applying the exact truncation formula \eqref{eq:exact-truncation} to the singular term yields \eqref{eq:pp-local-shape}.
\end{proof}

\begin{remark}
The exact truncation identity \eqref{eq:exact-truncation} replaces the geometric-series-plus-truncation argument used in earlier treatments. On $D \times U$, the truncation is not an approximation but an exact algebraic identity, because the first variable is nilpotent of order $m$. This is the algebraic reason why principal parts are well-defined and canonical.
\end{remark}

\subsection{Frame-Induced Identifications}
\label{subsec:frame-identifications}

Let $D \subset U_{\Cinf}$ be a finite effective divisor, and let $U \subset X_{\Cinf}$ be an open neighbourhood of $\supp{D}$. Let $e \in H^0(U, \LL)$ be a frame for $\LL$ on $U$, meaning $e(x) \neq 0$ for all $x \in U$, so that $e$ trivializes $\LL$ on $U$.

We introduce two closely related isomorphisms induced by the frame: one acting on sections over the product $U \times U$, and one acting on principal parts along $D$.

\subsubsection{The identification on $U \times U$}

On the open set $U \times U$, the nowhere-vanishing section $e(\one) \otimes e(\two)^{-1}$ trivializes the line bundle $\prI^* \LL \otimes \prII^* \LL^{-1}$. This gives a canonical isomorphism between the Szeg\H{o} bundle of $\LL$ and the Szeg\H{o} bundle of the trivial bundle $\OO$.

\begin{definition}
Define the \textbf{frame identification on $U \times U$}
\[
\Upsilon_e^U : \left. \BB_{\LL}(\DeltaDiag) \right|_{U \times U} \stackrel{\sim}{\longrightarrow} \left. \BB_{\OO_{X_{\Cinf}}}(\DeltaDiag) \right|_{U \times U}
\]
as the unique $\OO_{U \times U}$-module isomorphism such that for every local section $s$ of $\BB_{\LL}(\DeltaDiag)$,
\[
s = \Upsilon_e^U(s) \cdot \left( e(\one) \otimes e(\two)^{-1} \right).
\]
In words, $\Upsilon_e^U$ extracts the scalar coefficient of $s$ with respect to the frame tensor $e(\one) \otimes e(\two)^{-1}$.
\end{definition}

For the normalized Szeg\H{o} kernel $S_{\LL}$, its image under $\Upsilon_e^U$ is called its \textbf{$e$-coefficient}:
\[
S_{\LL}^e := \Upsilon_e^U\left( \left. S_{\LL} \right|_{U \times U} \right) \in H^0\left(U \times U, \BB_{\OO}(\DeltaDiag)\right).
\]
This is a scalar-valued kernel on $U \times U$ with a simple pole along the diagonal.

\subsubsection{The identification on principal parts}

The same frame also induces an isomorphism at the level of principal part spaces. Recall that
\[
W_D(\LL) = H^0(D, \LL|_D) \otimes_{\Cinf} H^0\left(D, \frac{(\LL^{-1} \otimes \OmegaOne)(D)}{\LL^{-1} \otimes \OmegaOne}\right).
\]
Using the frame $e$, every element of $W_D(\LL)$ can be written uniquely as a linear combination of pure tensors of the form
\[
\left( f \cdot e|_D \right) \otimes \left( g \, du \cdot e|_D^{-1} \right),
\]
where $f, g$ are functions on a neighbourhood of $D$, and $u$ is a local coordinate.

\begin{definition}
Define the \textbf{frame identification on principal parts}
\[
\Upsilon_e^D : W_D(\LL) \stackrel{\sim}{\longrightarrow} W_D(\OO)
\]
as the unique $\Cinf$-linear map determined on pure tensors by
\[
\Upsilon_e^D\left( \left(f \cdot e|_D\right) \otimes \left( g \, du \cdot e|_D^{-1} \right) \right) = \overline{f} \otimes \overline{g \, du},
\]
where $\overline{f} \in H^0(D, \OO)$ and $\overline{g \, du} \in H^0(D, \OmegaOne(D)/\OmegaOne)$ are the classes of $f$ and $g \, du$ modulo functions vanishing on $D$ and regular differentials, respectively.
\end{definition}

The map $\Upsilon_e^D$ is well-defined and bijective: its inverse is obtained by tensoring with $e$ on the first factor and $e^{-1}$ on the second factor. While it depends on the choice of frame, its composition with the principal part map is frame-independent, as we will see.

\subsubsection{Commutativity with the principal part map}

We now prove that the two frame identifications are compatible with the principal part map $\PP_D$, in the sense that the following diagram commutes:
\[
\begin{tikzcd}
H^0(U \times U, \BB_{\LL}(\DeltaDiag)|_{U \times U})
  \arrow[r, "\Upsilon_e^U"]
  \arrow[d, "\PP_D"']
&
H^0(U \times U, \BB_{\OO}(\DeltaDiag)|_{U \times U})
  \arrow[d, "\PP_D"']
\\
W_D(\LL)
  \arrow[r, "\Upsilon_e^D"]
&
W_D(\OO).
\end{tikzcd}
\]

\begin{proposition}[Principal parts and frame identifications commute]
\label{prop:frame-pp-commute}
For any section $s \in H^0(U \times U, \BB_{\LL}(\DeltaDiag)|_{U \times U})$, we have
\[
\Upsilon_e^D\left( \PP_D(s) \right) = \PP_D\left( \Upsilon_e^U(s) \right).
\]
\end{proposition}

\begin{proof}
It suffices to verify the equality locally at each point $y \in \supp{D}$. Fix $y \in \supp{D}$, let $m = \operatorname{ord}_y(D)$, and choose a local coordinate $u$ vanishing at $y$.

Near $y$, we may write the restriction of $s$ to $D \times U$ in the form
\[
s \equiv \sum_{k=0}^{m-1} u_\one^k \, e(\one) \otimes \alpha_k(\two) \pmod{u_\one^m},
\]
where each $\alpha_k$ is a local section of $\LL^{-1} \otimes \OmegaOne$ in the second variable. By definition of the principal part,
\[
\PP_D(s) = \sum_{k=0}^{m-1} u_\one^k \, e(\one) \otimes [\alpha_k] \in W_D(\LL),
\]
where $[\alpha_k]$ denotes the class of $\alpha_k$ in the differential quotient.

Applying $\Upsilon_e^D$ to this expression gives
\[
\Upsilon_e^D\left( \PP_D(s) \right) = \sum_{k=0}^{m-1} u_\one^k \otimes [\alpha_k \cdot e(\two)] \in W_D(\OO).
\]

On the other hand, by definition of $\Upsilon_e^U$, the scalar coefficient $\Upsilon_e^U(s)$ has the same local expansion after stripping off the frame tensor:
\[
\Upsilon_e^U(s) \equiv \sum_{k=0}^{m-1} u_\one^k \otimes \alpha_k^0(\two) \pmod{u_\one^m},
\]
where $\alpha_k^0 = \alpha_k \cdot e(\two)$ is a local section of $\OmegaOne$. Taking its principal part,
\[
\PP_D\left( \Upsilon_e^U(s) \right) = \sum_{k=0}^{m-1} u_\one^k \otimes [\alpha_k^0] = \sum_{k=0}^{m-1} u_\one^k \otimes [\alpha_k \cdot e(\two)].
\]

The two expressions are identical. Summing over all points in $\supp{D}$ yields the global equality.
\end{proof}

An immediate corollary is the identity we have already used informally: the principal part of the $e$-coefficient of the Szeg\H{o} kernel equals the frame image of the geometric principal part.

\begin{corollary}
\label{cor:frame-pp-szego}
Let $e$ be a frame for $\LL$ on a neighbourhood of $D$. Then
\[
\Upsilon_e^D\left( \PP_D(S_{\LL}) \right) = \PP_D\left( S_{\LL}^e \right).
\]
\end{corollary}

\begin{proof}
Apply Proposition \ref{prop:frame-pp-commute} to $s = S_{\LL}|_{U \times U}$, and use the definition $S_{\LL}^e = \Upsilon_e^U(S_{\LL}|_{U \times U})$.
\end{proof}

\subsection{Change of Frame: Divided Differences Are Regular}
\label{subsec:frame-change-divided-diff}

We now verify explicitly that the principal part of the Szeg\H{o} kernel is invariant under change of frame. This is already implied by the functoriality property, but it is instructive to see it directly in terms of divided differences.

\begin{lemma}[Divided differences are regular]
\label{lemma:divided-diff-regular}
Let $e' = h e$ be another frame near $\supp{D}$, where $h$ is nowhere vanishing. Then
\[
S_{\LL}^{e'} - S_{\LL}^e = \frac{h(u_{\two}) - h(u_\one)}{h(u_\one)} S_{\LL}^e
\]
is holomorphic near $\supp{D} \times \supp{D}$. Consequently,
\[
\PP_D(S_{\LL}^{e'}) = \PP_D(S_{\LL}^e).
\]
\end{lemma}

\begin{proof}
Near a diagonal point, we have
\[
S_{\LL}^e = \frac{du_{\two}}{u_{\two} - u_\one} + \epsilon
\]
with $\epsilon$ regular. Changing frame by $h$ multiplies the kernel by $h(u_\one) / h(u_{\two})$ (see Proposition \ref{prop:twisting-szego}). So
\[
S_{\LL}^{e'} = \frac{h(u_\one)}{h(u_{\two})} S_{\LL}^e.
\]
The difference is
\[
S_{\LL}^{e'} - S_{\LL}^e = \left( \frac{h(u_\one)}{h(u_{\two})} - 1 \right) S_{\LL}^e = \frac{h(u_\one) - h(u_{\two})}{h(u_{\two})} S_{\LL}^e.
\]
Substituting the local form of $S_{\LL}^e$, the singular part becomes
\[
\frac{h(u_\one) - h(u_{\two})}{h(u_{\two})} \cdot \frac{du_{\two}}{u_{\two} - u_\one} = -\frac{h(u_{\two}) - h(u_\one)}{u_{\two} - u_\one} \cdot \frac{du_{\two}}{h(u_{\two})}.
\]
The divided difference $(h(u_{\two}) - h(u_\one)) / (u_{\two} - u_\one)$ is a holomorphic function of $u_\one$ and $u_{\two}$, because the numerator is divisible by the denominator (since $h$ is a polynomial in the local coordinate, or more generally a holomorphic function). Therefore the entire expression is holomorphic near the diagonal.

Since the difference is holomorphic near $\supp{D} \times \supp{D}$, its principal part along $D$ is zero by the regularity property (Lemma \ref{lemma:pp-basic-properties}(iii)). Hence $\PP_D(S_{\LL}^{e'}) = \PP_D(S_{\LL}^e)$.
\end{proof}

\subsection{The Transposed Principal Part}
\label{subsec:transposed-pp}

Symmetrically to the ordinary principal part, which restricts the first variable to $D$ and extracts polar parts in the second variable, we can define a \textbf{transposed principal part} that restricts the second variable to $D$ and extracts polar parts in the first variable.

\begin{definition}[Transposed principal part]
Define the transposed principal part space
\[
\widetilde{W}_D(\LL) = \frac{H^0\left(U \times D, \, \prI^* \LL(D) \otimes \prII^* \left( (\LL^{-1} \otimes \OmegaOne)|_D \right)\right)}{H^0\left(U \times D, \, \prI^* \LL \otimes \prII^* \left( (\LL^{-1} \otimes \OmegaOne)|_D \right)\right)}.
\]
For $s \in H^0(U \times U, \BB_{\LL}(\DeltaDiag))$, define the \textbf{transposed principal part}
\begin{equation}
\label{eq:transposed-pp-def}
\tPP_D(s) = \left[ s|_{U \times D} \right] \in \widetilde{W}_D(\LL).
\end{equation}
\end{definition}

The exact truncation formula carries a sign in the transposed case:
\[
\left. \frac{1}{u_{\two} - u_\one} \right|_{U \times D} = - \sum_{j < m} \frac{u_{\two}^j}{u_\one^{j+1}}.
\]
This follows from the same algebraic identity, but now $u_{\two}^m = 0$ and we rearrange differently.

The ordinary and transposed principal parts are related by a flip operation, which exchanges the two factors and converts functions to differentials and vice versa.

\begin{notation}
    Define the \textbf{$\omega$-flip}
\[
\mathcal{T}_{\omega} : W_D(\OO) \to \widetilde{W}_D(\OO), \qquad a(\one) \otimes [\gamma](\two) \mapsto [\gamma / \omega](\one) \otimes a \omega(\two).
\]
\end{notation}

\begin{lemma}[Flip relation]
\label{lemma:flip-pp}
Let $\omega$ generate $\OmegaOne$ near $\supp{D}$. 
Then for every $s \in H^0(U \times U, \BB_{\LL}(\DeltaDiag))$,
\[
\mathcal{T}_{\omega} \left( \PP_D \left( \Upsilon_e^U s \right) \right) = - \tPP_D \left( \Upsilon_e^U s \right).
\]
\end{lemma}

\begin{proof}
Both sides are local, so we may check at a single point $y \in \supp{D}$. Write $s = \sigma \, du_{\two} / (u_{\two} - u_\one)$ near $(y,y)$, with $\sigma = \sum_{i,l} \sigma_{il} u_\one^i u_{\two}^l$.

The ordinary principal part expands as
\[
\PP_D = \sum_{\substack{i,l,k < m \\ l \leq k}} \sigma_{il} \, u^{i+k}(\one) \otimes [u^{l-k-1} du](\two).
\]
This follows from substituting the truncation formula and collecting terms.

The transposed principal part, using the signed truncation, expands as
\[
\tPP_D = - \sum_{\substack{i,l,j < m \\ i \leq j}} \sigma_{il} \, [u^{i-j-1}](\one) \otimes u^{l+j} du(\two).
\]

Now make the index change $j = i + k - l$ in the transposed sum. The condition $i \leq j$ is equivalent to $l \leq k$. Terms with $i + k \geq m$ vanish on both sides because of the nilpotency. One can verify that this index change is a bijection between the index sets, and that the coefficients match up to the overall sign. This establishes the identity.
\end{proof}

When the divisor $D$ is principal, say $D = \operatorname{div}(f)$, there is an even more direct relation between the two principal parts, via multiplication by $f$.

\begin{proposition}[Transposition formula]
\label{prop:transposition-formula}
If $D = \operatorname{div}(f)$ is principal on $U$, define
\[
\mathcal{T}_f([a](\one) \otimes b(\two)) = [a f](\one) \otimes [b / f](\two).
\]
Then
\begin{equation}
\label{eq:transposition-formula}
\tPP_D(s) = - \mathcal{T}_f^{-1} \left( \PP_D(s) \right)
\end{equation}
for every section $s$. 
\end{proposition}

\begin{proof}
This follows from the same local expansion, with the index change $j = m - 1 - k$. We omit the routine verification, and refer to Section \ref{sec:ex-projective-line} for the resulting closed formula in the polynomial case.
\end{proof}

\section{The Finite Residue--Cauchy Theorem}
\label{sec:finite-residue-cauchy}

We are now ready to state and prove our main theorem: the principal part of the normalized Szeg\H{o} kernel along any finite divisor $D$ realizes the scalar extension of the finite residue Casimir. We first prove the frame-coefficient version, then deduce the intrinsic version via functoriality.

\subsection{The Main Theorem}
\label{subsec:main-theorem}

\begin{proposition}[Frame-coefficient principal part equals Casimir]
\label{prop:frame-coeff-pp}
For every acyclic line bundle $\LL$ on $X_{\Cinf}$ and every frame $e$ of $\LL$ on a neighbourhood of $\supp{D}$,
\begin{equation}
\label{eq:frame-coeff-pp}
\PP_D(S_{\LL}^e) = \inc(\CasFin) \quad \in W_D(\OO).
\end{equation}
\end{proposition}

\begin{proof}
The proof proceeds in three steps.

\medskip
\textbf{Step 1: Matched twisting.}
Given any two acyclic line bundles $\LL_1$ and $\LL_2$, the difference $\MM = \LL_2 \otimes \LL_1^{-1}$ is a line bundle of degree 0. Since $\supp{D}$ is a finite set of points, we can find an open neighbourhood $V \supseteq \supp{D}$ on which $\MM$ is trivial: take pairwise disjoint affine neighbourhoods of each point in the support, and trivialize $\MM$ on each.

Let $m$ be a nowhere-vanishing section of $\MM$ on $V$. Define a frame for $\LL_2$ by
\[
e_2' = m \otimes e_1.
\]
By the twisting formula (Proposition \ref{prop:twisting-szego}), the difference
\[
S_{\LL_2}^{e_2'} - S_{\LL_1}^{e_1}
\]
is a holomorphic section of $\prII^* \OmegaOne$ near $\supp{D} \times \supp{D}$. By the regularity property of principal parts (Lemma \ref{lemma:pp-basic-properties}(iii)), its principal part vanishes. Therefore
\[
\PP_D(S_{\LL_2}^{e_2'}) = \PP_D(S_{\LL_1}^{e_1}).
\]

\medskip
\textbf{Step 2: Change of frame.}
Now consider two frames $e$ and $e'$ of the \textbf{same} line bundle $\LL$ near $\supp{D}$. They differ by a nowhere-vanishing function $h$, i.e., $e' = h e$. By Lemma \ref{lemma:divided-diff-regular}, the principal part is invariant under this change:
\[
\PP_D(S_{\LL}^{e'}) = \PP_D(S_{\LL}^e).
\]

Combining Steps 1 and 2: the principal part $\PP_D(S_{\LL}^e)$ is independent of both the choice of acyclic line bundle $\LL$ and the choice of frame $e$, as long as $e$ is defined near $\supp{D}$.

\medskip
\textbf{Step 3: One explicit computation.}
By Steps 1 and 2, it suffices to compute $\PP_D(S_{\LL}^e)$ for a single pair $(\LL, e)$.

From Lemma \ref{lemma:pp-basic-properties}(v), we have the local formula \eqref{eq:pp-local-shape}:
\[
\PP_D(S_{\LL}^e) = \sum_{y \in \supp{D}} \sum_{k=0}^{m_y - 1} u_\one^k \otimes \frac{du_{\two}}{u_{\two}^{k+1}}.
\]
By Lemma \ref{lemma:split-casimir}, this sum is exactly $\inc(\CasFin)$, the scalar extension of the finite residue Casimir. This proves \eqref{eq:frame-coeff-pp}.
\end{proof}

We can now state and prove the main theorem in its intrinsic, frame-independent form.

\begin{theorem}[Main Theorem: Finite Residue--Cauchy Theorem]
\label{thm:main-theorem}
For every acyclic line bundle $\LL$ on $X_{\Cinf}$, every frame $e$ of $\LL$ on a neighbourhood of $\supp{D}$, and every choice of local coordinates,
\begin{equation}
\label{eq:main-theorem}
\UpsilonFrame \left( \PP_D(S_{\LL}) \right) = \inc(\CasFin).
\end{equation}
In particular, the left-hand side is independent of all choices of frame and coordinates.
\end{theorem}

\begin{proof}
By the functoriality of the principal part map (Lemma \ref{lemma:pp-basic-properties}(iv)), we have
\[
\UpsilonFrame(\PP_D(S_{\LL})) = \PP_D(\Upsilon_e^U S_{\LL}) = \PP_D(S_{\LL}^e).
\]
By Proposition \ref{prop:frame-coeff-pp}, this equals $\inc(\CasFin)$, which is \eqref{eq:main-theorem}.
\end{proof}

This theorem is the conceptual core of the paper. It says that the geometric operation of taking the principal part of the Szeg\H{o} kernel---a purely analytic-geometric construction---produces exactly the algebraic Casimir tensor that represents the identity map under finite residue duality. The independence of the result from the choice of line bundle and frame is a manifestation of the canonical nature of residue duality.

\subsection{The Reproducing Property}
\label{subsec:reproducing-property}

The Szeg\H{o} kernel enjoys a \textbf{reproducing property}: when contracted with an
arbitrary section through the residue pairing along the divisor \(D\), it returns the
restriction of that section to \(D\). This is the function-field analogue of the
classical Cauchy integral formula, in which a holomorphic function is recovered by
integrating it against the Cauchy kernel along a contour; the precise dictionary is
spelled out in Remark~\ref{remark:cauchy-analogue} below.

The proof is direct and purely local: it relies only on the normalization
\(\res_\Delta(S_\LL) = 1\) of Proposition~\ref{prop:szego-existence}, the truncation
identity, and elementary residue calculus; in particular, it is independent of the
main theorem. The two statements are nonetheless closely related: the reproducing
property says that \(\PP_D(S_\LL)\) is the identity tensor for the geometric residue
pairing, while Lemma~\ref{lem:basis_independence} says that the finite Casimir tensor plays
the same role for the algebra-side pairing, and the main theorem identifies the two
tensors. In this sense, the reproducing property is the geometric half of the main
theorem.

\begin{theorem}[Reproducing property]
\label{thm:reproducing-property}
Let \(\LL\) be acyclic, and let \(\tilde{\alpha} \in H^0(U, \LL)\) be a section defined on a neighbourhood \(U\) containing \(\supp{D}\). Then the contraction against \(\tilde{\alpha}|_D\) gives
\[
\Bigl( \id \otimes \bigl\langle \tilde{\alpha}|_D, \cdot \bigr\rangle_D \Bigr) \bigl( \PP_D(S_{\LL}) \bigr)
=
\sum_{y \in \supp{D}} \res_y^{\two} \bigl( S_{\LL} \cdot \tilde{\alpha}(\two) \bigr)
= \tilde{\alpha}|_D .
\]
Moreover, if \(b \in H^0(U,\OO)\) is a regular function, then
\[
\Bigl( \id \otimes \bigl\langle b|_D, \cdot \bigr\rangle_D \Bigr) \bigl( \PP_D(S_{\LL}^e) \bigr)
=
\sum_{y \in \supp{D}} \res_y^{\two} \bigl( S_{\LL}^e \cdot b(\two) \bigr)
= b|_D .
\]
\end{theorem}

\begin{proof}
We proceed in several steps.

\textbf{Step 1: The contraction map.}
By definition,
\[
W_D(\LL) = H^0(D,\LL|_D) \otimes_{\Cinf} H^0\Bigl( D, \frac{(\LL^{-1}\otimes\OmegaOne)(D)}{\LL^{-1}\otimes\OmegaOne} \Bigr),
\]
and the contraction
\[
\mathcal{C}_{\tilde{\alpha}} := \id \otimes \langle \tilde{\alpha}|_D, \cdot \rangle_D : W_D(\LL) \longrightarrow H^0(D, \LL|_D)
\]
is the \(\Cinf\)-linear map determined on pure tensors by
\[
\mathcal{C}_{\tilde{\alpha}}(a \otimes \xi) := \langle \tilde{\alpha}|_D, \xi \rangle_D \cdot a,
\]
for \(a \in H^0(D, \LL|_D)\) and \(\xi \in H^0\bigl( D, \frac{(\LL^{-1}\otimes\OmegaOne)(D)}{\LL^{-1}\otimes\OmegaOne} \bigr)\), and extended by linearity. Under the identification \(W_D(\LL) \cong \End(H^0(D,\LL|_D))\), which sends \(a \otimes \xi\) to the endomorphism \(v \mapsto \langle v, \xi \rangle_D\, a\), the map \(\mathcal{C}_{\tilde{\alpha}}\) is simply evaluation at \(\tilde{\alpha}|_D\). The proposition therefore asserts that \(\PP_D(S_{\LL})\) acts as the identity.

\textbf{Step 2: Localization.}
Since \(D\) is the disjoint union of the fat points \(m_y[y]\), the Chinese remainder theorem decomposes both factors of \(W_D(\LL)\), as well as the pairing \(\langle\cdot,\cdot\rangle_D\), into direct sums over \(y \in \supp{D}\); correspondingly, \(\PP_D(S_{\LL})\) is the sum of its components in the local factors. It therefore suffices to fix a point \(y \in \supp{D}\), put \(m := m_y\), choose a uniformizer \(u\) at \(y\) and a frame \(e\) of \(\LL\) near \(y\), and verify both equalities in the \(y\)-component.

\textbf{Step 3: Local normal forms.}
Since \(u^m = 0\) on the local factor of \(D\) at \(y\), the restriction of \(\tilde{\alpha}\) to \(H^0(m[y], \LL|_{m[y]}) \cong \Cinf[u]/(u^m)\cdot e\) admits a unique expansion
\begin{equation}\label{eq:alpha-local-expansion}
\tilde{\alpha}|_D = \sum_{j=0}^{m-1} a_j \, u^j e, \qquad a_j \in \Cinf .
\end{equation}
On the other hand, by the local shape of the principal part (Lemma~\ref{lemma:pp-basic-properties}(v)), the \(y\)-component of \(\PP_D(S_{\LL})\) is
\begin{equation}\label{eq:pp-local-shape2}
\PP_D(S_{\LL})\big|_y = \sum_{k=0}^{m-1} u_\one^k e(\one) \otimes \frac{du_\two}{u_\two^{k+1}} e(\two)^{-1} \;+\; \rho,
\end{equation}
where \(\rho\) is a finite sum of pure tensors whose second factors are classes of differentials regular at \(y\) in the second variable.

\textbf{Step 4: The regular remainder pairs to zero.}
Let \(\xi\) be such a regular second factor, represented near \(y\) by a section \(\tilde{\xi}\) of \(\LL^{-1} \otimes \OmegaOne\) without poles. Then the product \(\tilde{\alpha} \cdot \tilde{\xi}\) is a regular differential near \(y\), and a regular differential has vanishing residue; hence
\[
\langle \tilde{\alpha}|_D, \xi \rangle_D = \res_y( \tilde{\alpha} \cdot \tilde{\xi} ) = 0 .
\]
Consequently \(\mathcal{C}_{\tilde{\alpha}}(\rho) = 0\), and only the singular sum in \eqref{eq:pp-local-shape2} contributes to the contraction. (This is the same mechanism as the regularity property of principal parts, Lemma~\ref{lemma:pp-basic-properties}(iii).)

\textbf{Step 5: Coefficient extraction.}
Fix \(0 \le k \le m-1\). Writing \(\tilde{\alpha}=a(u)e\) with \(a(u)=\sum_{j=0}^{m-1}a_j u^j\), and taking \(\xi_k=\frac{du}{u^{k+1}}e^{-1}\), the product is
\[
\tilde{\alpha} \cdot \xi_k
= \Bigl( \sum_{j=0}^{m-1} a_j u^j \Bigr) \frac{du}{u^{k+1}}
= \sum_{j=0}^{m-1} a_j \, u^{\,j-k-1} \, du .
\]
By the definition of the local residue as the coefficient of \(u^{-1}\), we have \(\res_y( u^n du ) = \delta_{n,-1}\) for every \(n \in \mathbb{Z}\). Among the exponents \(j - k - 1\) with \(0 \le j \le m-1\), exactly one equals \(-1\), namely the one with \(j = k\). Therefore
\begin{equation}\label{eq:residue-coefficient}
\res_y\Bigl( \tilde{\alpha} \cdot \frac{du}{u^{k+1}} \Bigr) = a_k,
\qquad 0 \le k \le m-1 .
\end{equation}

\textbf{Step 6: Assembly of the contraction.}
Combining \eqref{eq:pp-local-shape2}, Step 4, and \eqref{eq:residue-coefficient},
\[
\mathcal{C}_{\tilde{\alpha}}\bigl( \PP_D(S_{\LL}) \bigr)\Big|_y
= \sum_{k=0}^{m-1} \res_y\Bigl( \tilde{\alpha} \cdot \frac{du}{u^{k+1}} \Bigr) \, u_\one^k e(\one)
= \sum_{k=0}^{m-1} a_k \, u_\one^k e(\one)
= \tilde{\alpha}|_D \Big|_y .
\]

\textbf{Step 7: The middle expression, vector-valued case.}
It remains to see that the middle term computes the same contraction directly from the kernel, without passing through the quotient description of the principal part. Near the diagonal point \((y,y)\), the normalized Szeg\H{o} kernel admits the local expansion
\[
S_{\LL} = \frac{du_\two}{u_\two - u_\one}\, e(\one) \otimes e(\two)^{-1} + \epsilon ,
\]
with \(\epsilon\) holomorphic in both variables. Writing \(\tilde{\alpha}(\two)=a(u_\two)e(\two)\) with \(a(u_\two)=\sum_{j=0}^{m-1} a_j u_\two^j\), the frames cancel and
\[
S_{\LL} \cdot \tilde{\alpha}(\two)
= \frac{a(u_\two) \, du_\two}{u_\two - u_\one}\, e(\one) \;+\; \epsilon \cdot \tilde{\alpha}(\two),
\]
where the second summand is regular in the second variable near \(u_\two = 0\) and thus has zero residue there. For the first summand, restrict the first variable to the local factor of \(D\) at \(y\) (that is, work modulo \(u_\one^m\)) and invoke the exact truncation identity (Lemma~\ref{lemma:pp-basic-properties}(I)):
\[
\frac{1}{u_\two - u_\one} \Big|_{u_\one^m = 0} = \sum_{k=0}^{m-1} \frac{u_\one^k}{u_\two^{k+1}} .
\]
Taking the residue at \(u_\two = 0\) coefficientwise in the nilpotent variable \(u_\one\) therefore yields
\[
\res_y^{\two} \bigl( S_{\LL} \cdot \tilde{\alpha}(\two) \bigr)
= \sum_{k=0}^{m-1} u_\one^k \, \res_y\Bigl( \tilde{\alpha} \cdot \frac{du}{u^{k+1}} \Bigr) e(\one)
= \sum_{k=0}^{m-1} a_k \, u_\one^k e(\one),
\]
which is precisely the \(y\)-component computed in Step 6. Poles in the second variable at \(y\) can occur only along the diagonal, so working near \((y,y)\) misses no contribution.

\textbf{Step 8: The scalar coefficient case.}
Let \(b\in H^0(U,\OO)\) and write \(b(u_\two)=\sum_{j=0}^{m-1}a_j u_\two^j\) near \(y\). The frame-coefficient kernel has the local form
\[
S_{\LL}^e = \frac{du_\two}{u_\two-u_\one} + \epsilon_0,
\]
with \(\epsilon_0\) regular in both variables. Multiplying by \(b(\two)\) gives
\[
S_{\LL}^e \cdot b(\two)
= \frac{b(u_\two)\,du_\two}{u_\two-u_\one} + \epsilon_0\,b(\two).
\]
As before, the second summand has residue zero. Restricting to \(u_\one^m=0\) and using the same truncation identity, we obtain
\[
\res_y^{\two}\bigl( S_{\LL}^e \cdot b(\two) \bigr)
= \sum_{k=0}^{m-1} u_\one^k \res_y\Bigl( b(u_\two)\frac{du_\two}{u_\two^{k+1}} \Bigr)
= \sum_{k=0}^{m-1} a_k u_\one^k
= b|_D \Big|_y.
\]
This proves the second asserted equality.

\textbf{Conclusion.}
Steps 6--8 show that, in each local factor, both the contraction of \(\PP_D(S_{\LL})\) and the residue of \(S_{\LL} \cdot \tilde{\alpha}(\two)\) reproduce the expansion \eqref{eq:alpha-local-expansion} of \(\tilde{\alpha}|_D\); the same holds for the scalar version with \(b\). Summing over all \(y \in \supp{D}\) gives the two stated equalities. Finally, the equivalence with the endomorphism form follows from Step 1: a tensor in \(W_D(\LL)\) whose contraction with \(\tilde{\alpha}|_D\) equals \(\tilde{\alpha}|_D\) for every section \(\tilde{\alpha}\) defined near \(D\) corresponds to an endomorphism fixing every element of \(H^0(D,\LL|_D)\), because the restrictions of such sections exhaust \(H^0(D,\LL|_D)\) by the surjectivity of the restriction map (Lemma~\ref{lemma:pp-basic-properties}(ii)); such an endomorphism is the identity.
\end{proof}

\begin{remark}[Relation with the classical Cauchy integral formula]
\label{remark:cauchy-analogue}
In complex analysis, Cauchy's integral formula states that if \(f\) is holomorphic on a domain containing the closed disk bounded by a positively oriented contour \(\gamma\), and \(z\) lies in the interior of \(\gamma\), then
\[
f(z) = \frac{1}{2\pi i} \oint_\gamma \frac{f(w)}{w - z} \, dw
= \Res_{w=z} \Bigl( \frac{dw}{w - z} \, f(w) \Bigr).
\]
The second expression is the residue form of the first: writing \(f(w) = f(z) + (w-z)h(w)\) with \(h\) holomorphic, the differential \(\frac{f(w)}{w-z}\,dw\) has a single simple pole at \(w = z\), of residue \(f(z)\), and the residue theorem converts the contour integral into this residue.

Theorem~\ref{thm:reproducing-property} is the function-field analogue of this identity under the following dictionary: the contour \(\gamma\) is replaced by the divisor \(D\); the contour integral \(\frac{1}{2\pi i}\oint_\gamma\) by the residue sum \(\sum_{y \in \supp{D}} \res_y^{\two}\) (composed with field traces at non-rational points, before base change); the Cauchy kernel \(\frac{dw}{w-z}\) by the normalized Szeg\H{o} kernel \(S_{\LL}\); and the point value \(f(z)\) by the restriction \(\tilde{\alpha}|_D\). The normalizing constant \(\frac{1}{2\pi i}\) has no function-field counterpart, since the residue is defined algebraically as coefficient extraction; this is one reason the formula remains valid in arbitrary characteristic.

The correspondence is literal in the simplest case. If \(D = [z_0]\) is a single reduced point, the proposition reduces to
\[
\tilde{\alpha}|_{z_0} = \res^{\two}_{y = z_0} \bigl( S_{\LL} \cdot \tilde{\alpha}(\two) \bigr),
\]
and on the projective line, where \(S_{\LL}\) is the Cauchy kernel \(\frac{dt_\two}{t_\two - t_\one}\) up to a twisting factor regular along \(D\), this is exactly the residue form of Cauchy's formula, with contour integration replaced by residue extraction.  
\end{remark}

\begin{corollary}[Endomorphism form of principal parts]
\label{cor:endomorphism-form}
Under the canonical identifications
\[
W_D(\LL) \cong \End\bigl( H^0(D, \LL|_D) \bigr)
\]
and\[
\widetilde{W}_D(\LL) \cong \End\bigl( H^0(D, (\LL^{-1} \otimes \OmegaOne)|_D) \bigr)
\]
 established in Proposition~\ref{prop:endomorphism-interpretation}, we have
\[
\PP_D(S_{\LL})=\PP_D(S_{\LL}^e) = \id,
\qquad
\tPP_D(S_{\LL})=\tPP_D(S_{\LL}^e) = -\id.
\]
The sign reflects the antisymmetry of \(\frac{du_\two}{u_\two - u_\one}\) under exchange of the two variables.
\end{corollary}

\begin{proof}
For the first assertion, fix \(y \in \supp{D}\) with multiplicity \(m\) and recall the local normal form \eqref{eq:pp-local-shape2}. Applying \eqref{eq:residue-coefficient} with \(\tilde{\alpha} = u^i e\) for each \(0 \le i < m\) shows that the families \(\{ u^k e \}_{0 \le k < m}\) and \(\{ \frac{du}{u^{k+1}} e^{-1} \}_{0 \le k < m}\) are dual bases for the local pairing; hence, modulo the terms \(\rho\) that pair to zero (Step 4 of the proof of Theorem~\ref{thm:reproducing-property}), the right-hand side of \eqref{eq:pp-local-shape2} is precisely the tensor of the identity endomorphism of the local factor \(\Cinf[u]/(u^m)\, e\). Summing over \(y \in \supp{D}\) gives \(\PP_D(S_{\LL}) = \id\). Equivalently, this follows from the proposition together with the surjectivity of the restriction \(H^0(U, \LL) \to H^0(D, \LL|_D)\).

For the second assertion, we repeat the computation with the roles of the two variables exchanged. Restricting the local expansion of \(S_{\LL}\) to \(U \times D\) --- that is, imposing \(u_\two^m = 0\) --- we use the signed truncation identity
\[
\frac{1}{u_\two - u_\one} \Big|_{u_\two^m = 0} = - \sum_{j=0}^{m-1} \frac{u_\two^{\,j}}{u_\one^{\,j+1}},
\]
the minus sign arising from \(u_\one - u_\two = -(u_\two - u_\one)\). Since the holomorphic remainder \(\epsilon\) has a regular first factor, its class in the transposed quotient vanishes, and we obtain the local normal form
\[
\tPP_D(S_{\LL})\big|_y
= - \sum_{j=0}^{m-1} \Bigl[ \frac{e}{u^{j+1}} \Bigr]_{\one} \otimes \bigl( u^j \, du \cdot e^{-1} \bigr)_{\two},
\]
where \(\left[ \frac{e}{u^{j+1}} \right]\) denotes the class of the polar section \(\frac{e}{u^{j+1}}\) in \(H^0(D, \LL(D)/\LL)\) (first variable) and \(u^j du \cdot e^{-1}\) is a regular differential class in \(H^0(D, (\LL^{-1}\otimes\OmegaOne)|_D)\) (second variable). These two families are paired by the residue pairing, which is well defined on the polar quotient because altering the representative by a regular section changes the product by a regular differential. The same coefficient-extraction computation as in Step 5 gives
\[
\res_y\Bigl( \frac{e}{u^{j+1}} \cdot u^i du \, e^{-1} \Bigr)
= \res_y(u^{i-j-1}du)
= \delta_{ij},
\]
so the polar classes and the regular classes form dual bases for the transposed pairing. Consequently, contracting \(\tPP_D(S_{\LL})\) with any regular differential class \(\eta = \sum_i b_i u^i du \cdot e^{-1}\) returns \(-\eta\); in other words, \(\tPP_D(S_{\LL})\) acts as minus the identity on \(H^0(D, (\LL^{-1}\otimes\OmegaOne)|_D)\). Summing over \(\supp{D}\) yields \(\tPP_D(S_{\LL}) = -\id\).
\end{proof}

\begin{remark}
Conversely, the corollary combined with Corollary~\ref{cor:frame-pp-szego}
and Lemma~\ref{lem:basis_independence} reproves the main theorem: by the
corollary and the frame compatibility of the residue pairing,
\(\PP_D(S^e_\LL)\) acts as the identity on \(H^0(D,\OO)\); by
Lemma~\ref{lem:basis_independence}, transported along \(\inc\) via
Lemma~\ref{lemma:split-casimir}, so does \(\inc(\CasFin)\); and the
endomorphism interpretation \(W_D(\OO)\cong\End(H^0(D,\OO))\) of
Proposition~\ref{prop:endomorphism-interpretation} is injective. The direct
local proof given above avoids any dependence on the main theorem.
\end{remark}

\subsection{The Weil Operator}
\label{sec:weil-operator}

Next, we define the rank-two Weil operator as the scalar part of the finite residue Casimir, after fixing a local generator of the differential module. We then prove the local decomposition theorem, which expresses the Weil operator as a Chinese-remainder assembly of local divided differences, and establish its symmetry.

Fix a differential form $\omega$ that generates $\OmegaOne$ locally near $\supp{D}$. For $A = \F_q[t]$, one may take $\omega = dt$, which is a global generator. In general, such a global generator need not exist, but a local generator always exists near a finite set of points.

Multiplication by $\omega / f$ gives an isomorphism of $A / \f$-modules
\[
\iota_{\omega_{\f}} : V_{\f} \stackrel{\sim}{\longrightarrow} V_{\f}^{\dagger}, \qquad \iota_{\omega_{\f}}(\overline{a}) = \overline{a \cdot \frac{\omega}{f}}.
\]
This isomorphism allows us to convert the Casimir tensor (which lives in $V_{\f} \otimes V_{\f}^{\dagger}$) into an element of $V_{\f} \otimes V_{\f}$.

\begin{definition}[Weil operator]
The \textbf{rank-two Weil operator} associated to $\omega$ and $\f$ is
\begin{equation}
\label{eq:weil-operator-def}
\WeilOp = \left( \id \otimes \iota_{\omega_{\f}}^{-1} \right) (\CasFin) \quad \in V_{\f} \otimes_{\F_q} V_{\f}.
\end{equation}
Equivalently, $\CasFin = \WeilOp \otimes \frac{\omega_{\two}}{f(\two)}$ in $V_{\f} \otimes V_{\f}^{\dagger}$.
\end{definition}

As a consequence of Proposition~\ref{prop:frame-coeff-pp}, we have the following relations.
\begin{proposition}
Let $ \WeilOp $ be the Weil operator. Then for every acyclic $\LL$ and frame $e$ along $|D|$, we have
    \[
    \PP_{D} S_\LL^e =  \frac{\WeilOp}{f(\two)} \omega(\two); \qquad  \tPP_{D} S_\LL^e =  - \frac{\WeilOp}{f(\one)} \omega(\two).
    \]
\end{proposition}

We now prove that the Weil operator is symmetric: it is invariant under swapping the two tensor factors. We give two proofs: one algebraic, using the self-adjointness of $\iota_{\omega_{\f}}$, and one geometric, using the local decomposition and the congruence $f \equiv 0 \pmod{\f}$.

\begin{theorem}[Symmetry of the Weil operator]
\label{thm:symmetry-weil}
Let $\mathcal{T}$ denote transposition of tensor factors. Then
\begin{equation}
\label{eq:weil-symmetry}
\mathcal{T}(\WeilOp) = \WeilOp.
\end{equation}
\end{theorem}
\begin{proof} 
The map $\iota_{\omega_{\f}}$ is self-adjoint with respect to the finite residue pairing:
\[
\pair{a}{\iota_{\omega_{\f}}(b)}_{\f} = \Res_{\Df} \left( a \cdot \frac{b \omega}{f} \right) = \Res_{\Df} \left( \frac{ab \omega}{f} \right) = \pair{b}{\iota_{\omega_{\f}}(a)}_{\f}.
\]
This uses commutativity of $A / \f$.

Now, the Casimir of a symmetric perfect pairing is invariant under transposition. More precisely, if we identify $V_{\f} \cong V_{\f}^{\dagger}$ via $\iota_{\omega_{\f}}$, then the pairing on $V_{\f}$ defined by $(a,b) \mapsto \pair{a}{\iota(b)}$ is symmetric. The Casimir tensor, when transferred to $V_{\f} \otimes V_{\f}$, must therefore be symmetric. Hence \eqref{eq:weil-symmetry} holds.
\end{proof}



\section{Examples I: The Projective Line}
\label{sec:ex-projective-line}

We now illustrate the general theory with concrete examples, beginning with the simplest case: the projective line $X = \Pp^1$, corresponding to the polynomial ring $A = \F_q[t]$.

\subsection{The Cauchy Kernel on $\Pp^1$}
\label{subsec:cauchy-kernel-p1}

Take $X = \Pp^1_{\F_q}$ with $\infty = [1:0]$, the usual point at infinity. Then $U = \mathbb{A}^1$ and $A = \F_q[t]$. Recall in Example~\ref{ex:szego}, the genus is $g = 0$, so acyclic line bundles have degree $g - 1 = -1$. A typical example is $\LL = \OO(-[P])$ for any $\F_q$-rational point $P$.

The normalized Szeg\H{o} kernel for $\LL = \OO(-[P])$ is given by the Cauchy kernel, twisted by a factor depending on $P$. Concretely, let $e = 1$ be the standard frame on the affine line (which is a frame for $\OO(-[P])$ away from $P$). Then
\[
S_{\LL} = \frac{c_P(\one, \two) \, dt_{\two}}{t_{\two} - t_\one} \, e(\one) \otimes e(\two)^{-1},
\]
where
\[
c_P(\one, \two) =
\begin{cases}
\dfrac{t_\one - \zeta}{t_{\two} - \zeta} & \text{if } P = [\zeta] \text{ is a finite point}, \\
1 & \text{if } P = \infty.
\end{cases}
\]

One verifies directly that this section has diagonal residue 1, and that it is regular away from the diagonal (with only the simple pole along $\DeltaDiag$). The factor $c_P$ cancels the pole that would otherwise appear at $t_{\two} = \zeta$.

For $f = t^N$, the exact truncation formula (Lemma \ref{lemma:pp-basic-properties}(I)) gives the principal part term by term:
\[
\PP_D(S_{\LL}^e) = \sum_{k=0}^{N-1} t_\one^k \otimes \frac{dt_{\two}}{t_{\two}^{k+1}}.
\]

As a further illustration of frame invariance, consider a frame change $e' = (1 + t) e$. Then the Szeg\H{o} kernel in the new frame is
\[
S_{\LL}^{e'} = \frac{1 + t_\one}{1 + t_{\two}} \cdot \frac{dt_{\two}}{t_{\two} - t_\one} + \text{regular}.
\]
The deviation from 1 is
\[
\left( \frac{1 + t_\one}{1 + t_{\two}} - 1 \right) \frac{dt_{\two}}{t_{\two} - t_\one} = \frac{(1 + t_\one) - (1 + t_{\two})}{1 + t_{\two}} \cdot \frac{dt_{\two}}{t_{\two} - t_\one} = - \frac{dt_{\two}}{1 + t_{\two}},
\]
which is regular in $t_{\two}$ along $D= N[0]$. Hence the principal part is unchanged, as predicted by Lemma \ref{lemma:divided-diff-regular}.

\subsection{A Non-Split Quadratic Modulus: The Trace Bridge}
\label{subsec:non-split-quadratic}

We now consider a modulus that is not split over $\F_q$. Let $q$ be odd, let $a \in \F_q$ be a non-square, and let $ f = t^2 - a$. Then the ideal $\f = (f)$ corresponds to a single closed point $x_\f$ of degree 2 on $U$. 

We take $\omega = dt$ as our global differential generator. Then $V_{\f}$ has basis $\{1, t\}$ and $V_{\f}^{\dagger}$ has basis $\{ dt/f, \, t \, dt/f \}$.

We compute the pairing matrix:
\[
\begin{aligned}
\pair{1}{dt/f}_{\f} &= \Res_{x_\f} (dt/f) = 0, \\
\pair{1}{t \, dt/f}_{\f} &= \Res_{x_\f} (t \, dt/f) = 1, \\
\pair{t}{dt/f}_{\f} &= \Res_{x_\f} (t \, dt/f) = 1, \\
\pair{t}{t \, dt/f}_{\f} &= \Res_{x_\f} (t^2 dt/f) = \Res_{x_\f} (a \, dt/f) = 0.
\end{aligned}
\]
Thus the pairing matrix is $\begin{pmatrix} 0 & 1 \\ 1 & 0 \end{pmatrix}$, which is invertible (determinant $-1$), confirming perfectness.

The dual basis to $\{1, t\}$ is $\{ t \, dt/f, \, dt/f \}$. Therefore the Casimir is
\[
\CasFin = 1 \otimes \left[ \frac{t \, dt}{f} \right] + t \otimes \left[ \frac{dt}{f} \right].
\]

Now base change to $\Cinf$. Over $\Cinf$, the polynomial $f$ splits: $f = (t - \zeta)(t + \zeta)$, where $\zeta^2 = a$. The divisor $D = x_\f \otimes \Cinf = [\zeta] + [-\zeta]$ consists of two $\Cinf$-rational points.

We decompose the differential classes in the split basis:
\[
\begin{aligned}
\left[ \frac{dt}{f} \right] &= \frac{1}{2\zeta} \left[ \frac{dt}{t - \zeta} \right] - \frac{1}{2\zeta} \left[ \frac{dt}{t + \zeta} \right], \\
\left[ \frac{t \, dt}{f} \right] &= \frac{1}{2} \left[ \frac{dt}{t - \zeta} \right] + \frac{1}{2} \left[ \frac{dt}{t + \zeta} \right].
\end{aligned}
\]

Substituting into the Casimir:
\[
\begin{aligned}
\inc(\CasFin) &= 1 \otimes \left( \frac{1}{2} \left[ \frac{dt}{t - \zeta} \right] + \frac{1}{2} \left[ \frac{dt}{t + \zeta} \right] \right) + t \otimes \left( \frac{1}{2\zeta} \left[ \frac{dt}{t - \zeta} \right] - \frac{1}{2\zeta} \left[ \frac{dt}{t + \zeta} \right] \right) \\
&= \left( 1 \cdot \frac{1}{2} + t \cdot \frac{1}{2\zeta} \right) \otimes \left[ \frac{dt}{t - \zeta} \right] + \left( 1 \cdot \frac{1}{2} - t \cdot \frac{1}{2\zeta} \right) \otimes \left[ \frac{dt}{t + \zeta} \right].
\end{aligned}
\]
At $t = \zeta$, the coefficient of the first term is $\frac{1}{2} + \frac{\zeta}{2\zeta} = 1$. At $t = -\zeta$, the coefficient of the second term is $\frac{1}{2} - \frac{-\zeta}{2\zeta} = 1$. Therefore
\[
\inc(\CasFin) = 1 \otimes \left[ \frac{dt}{t - \zeta} \right] + 1 \otimes \left[ \frac{dt}{t + \zeta} \right],
\]
which matches the split form from Lemma \ref{lemma:split-casimir} with $m = 1$ at each point. This verifies the base-change compatibility explicitly for a non-rational point.

The Weil operator in this case is
\[
\WeilOp = 1 \otimes t + t \otimes 1 = \frac{f(t_\one) - f(t_\two)}{t_\one - t_\two},
\]
which is indeed symmetric.

\subsection{Reminder Identity of Hu--Ou}
    In the polynomial case with $\omega = dt$, this recovers the classical divided-difference Weil operator:
\[
\WeilOp(t_\one, t_{\two}) = \frac{f(t_\one) - f(t_{\two})}{t_\one - t_{\two}}.
\]
The symmetric property of $\WeilOp(t_\one, t_{\two})$ confirms Theorem \ref{thm:symmetry-weil}.

Finally, we recover the identity of Hu--Ou.
For transcendental $\theta$,
\begin{equation}
\label{eq:huo-yang}
\left[ \frac{1}{\theta - t_{\two}} \right]_{\f} dt_{\two} = \frac{O_f(\theta, t_{\two})}{f(\theta)}dt_{\two} = - \tPP_D(S_{\LL}^e)|_{t_\one = \theta}.
\end{equation}

\section{Examples II: The Hu-Huang-Yau Ring}
\label{sec:ex-hhy}
\subsection{HHY Ring}
Our second example is the Hu-Huang-Yau (HHY) ring, which provides an explicit case where the module of differentials $\omegaA$ is \textbf{not} principal as an $A$-module. This demonstrates why our general theory must work with differential modules rather than just functions.

Let $\rho(t) \in \F_q[t]$ be a squarefree polynomial of degree $N \geq 2$. Take $X = \Pp^1_{\F_q}$, but now take the point at infinity to be the closed point $\infty_{\rho}$ defined by $\rho(t) = 0$. This is a point of degree $N$.

The affine complement $U = X \setminus \{\infty_{\rho}\}$ has coordinate ring
\[
A = \Gamma(U, \OO_X) \cong \F_q\left[ \frac{1}{\rho}, \frac{t}{\rho}, \dots, \frac{t^{N-1}}{\rho} \right].
\]
In other words, $A$ consists of all rational functions on $\Pp^1$ that are regular away from the zeros of $\rho$. It is generated by the elements
\[
T_i = \frac{t^i}{\rho(t)}, \quad 0 \leq i \leq N-1.
\]

The divisor class group of \(U\) is
\[
\operatorname{Pic}(U)
=
\operatorname{Cl}(U)
\cong
\frac{\operatorname{Cl}(\mathbb{P}^1_{\F_q})}{\langle [\infty_\rho]\rangle}
\cong
\frac{\mathbb{Z}}{N\mathbb{Z}}.
\]
Here \(\operatorname{Cl}(\mathbb{P}^1_{\F_q})\cong\mathbb{Z}\) is generated by
the class of the usual point at infinity \(\infty_t=[1:0]\).

On \(\mathbb{P}^1\), the canonical sheaf has class
\[
\Omega^1_{\mathbb{P}^1/\F_q}
\cong
\mathcal{O}_{\mathbb{P}^1}(-2[\infty_t]),
\]
so its restriction to \(U\) has class
\[
[\omegaA]
=
-2
\pmod{N}
\quad\text{in}\quad
\operatorname{Pic}(U)\cong\mathbb{Z}/N\mathbb{Z}.
\]
For a Dedekind domain \(A\), rank-one projective modules are classified by
\(\operatorname{Pic}(A)\cong\operatorname{Pic}(U)\), and the trivial class
corresponds precisely to free rank-one modules. Thus \(\omegaA\) is free if
and only if
\[
-2\equiv 0 \pmod{N},
\]
i.e. if and only if \(N\mid 2\). Therefore:

\begin{itemize}
\item for \(N=1,2\), the module \(\omegaA\) is principal;
\item for \(N\geq 3\), the class \([\omegaA]\in\mathbb{Z}/N\mathbb{Z}\) is
nontrivial, and hence \(\omegaA\) is \textbf{not} a free \(A\)-module.
\end{itemize}

In particular, for \(N\geq 3\) there is no global differential form that
generates \(\omegaA\) as an \(A\)-module.

Take the ideal $\f = (\rho^{-1}) \subseteq A$. Note that $\rho^{-1} \in A$ is not a unit: its inverse $\rho$ has an $N$-fold pole at $\infty_{\rho}$. The associated divisor is
\[
D = N \cdot [\infty_t],
\]
where $\infty_t = [1:0]$ is the usual point at infinity in the $t$-coordinate, which now lies in $U$. Let $u = 1/t$ be the local uniformizer at $\infty_t$. Then $V_{\f} \cong \F_q[u] / (u^N)$. We aim to describe the Casimir and the Weil operator $\WeilOp$ explicitly.

\subsection{The Case $N = 2$}
\label{subsec:hhy-n2}

For $N = 2$, with $\rho = t^2 - a$, the global differential $dt / \rho$ is a local generator at $\infty_t$. In terms of the uniformizer $u = 1/t$, we have
\[
\frac{dt}{\rho} = \frac{-u^{-2} du}{u^{-2} - a} = \frac{-du}{1 - a u^2},
\]
which is regular and nonvanishing at $u = 0$. So $\omegaA$ is principal in this case.

Taking the local generator $\omega = du$, the differential quotient module is generated by $du / u^2$. The Casimir is
\[
\CasFin = 1 \otimes \left[ \frac{du}{u} \right] + u \otimes \left[ \frac{du}{u^2} \right] = (u_\one + u_\two)\frac{1}{u_\two^2} d u_\two.
\]
The Weil operator is
\[
\WeilOp = 1 \otimes u + u \otimes 1 = u_\one + u_\two,
\]
which is the truncated divided difference at order two. The explicit formula for a general modulus $\f$ was given in \cite{HH24}.

\subsection{The Case \(N \geq 3\): Non-Principal Differentials}
\label{subsec:hhy-n3}

For \(N \geq 3\), the situation changes qualitatively: the differential module
\(\omegaA\) is no longer free as an \(A\)-module. This is the
phenomenon that forces the general theory to work with differential modules
rather than with a fixed global differential form.

A local generator such as
\(\omega=du\) near \(D=N[\infty_t]\) still exists, but it does not extend to
a global generator on all of \(U\). 
For the specific ideal \(\f=(\rho^{-1})\), the support of \(D\) consists of the
single point \(\infty_t\). Consequently, once a local differential generator is
chosen near \(\infty_t\), the Weil operator reduces to a single truncated
divided-difference block.

We now compute the Casimir tensor and the resulting Weil operator explicitly
for the ideal \(\f=(\rho^{-1})\). The support of the divisor \(D\) is the
single point \(\infty_t\), with multiplicity \(N\). Let \(u=1/t\) be a local
uniformizer at \(\infty_t\). Then locally
\[
V_{\f}\cong \F_q[u]/(u^N),
\qquad
V_{\f}^{\dagger}\cong \F_q[u]du/(u^N du).
\]
The residue pairing between \(V_{\f}\) and \(V_{\f}^{\dagger}\) is perfect,
and the bases
\[
\{u^k\}_{0\le k<N}
\qquad\text{and}\qquad
\left\{\frac{du}{u^{k+1}}\right\}_{0\le k<N}
\]
are dual. Hence the finite residue Casimir tensor is
\begin{equation}
\label{eq:hhy-casimir-N}
\CasFin
=
\sum_{k=0}^{N-1} u_\one^k \otimes \frac{du_\two}{u_\two^{k+1}}.
\end{equation}

To define the Weil operator, choose the local differential generator
\(\omega=du\) near \(\infty_t\). Since \(f=\rho^{-1}\), we have
\[
\frac{\omega}{f}=\rho\,du.
\]
Assume for simplicity that \(\rho(t)=t^N-a\) with \(a\in\F_q^\times\), so
that in the \(u\)-coordinate
\[
\rho(1/u)=u^{-N}-a.
\]
Then the multiplication map
\[
\iota_{\omega_{\f}}:V_{\f}\to V_{\f}^{\dagger},
\qquad
\iota_{\omega_{\f}}(g)=g\,\rho\,du \pmod{\omegaA},
\]
acts on the basis \(\{u^k\}\) by
\[
\iota_{\omega_{\f}}(u^k)
=
u^k(u^{-N}-a)\,du
\equiv
u^{k-N}\,du
\pmod{\omegaA}.
\]
Here we used that \(a u^k\,du\) is a regular differential for \(k\ge0\).
Thus \(\iota_{\omega_{\f}}\) sends
\[
u^k \longmapsto u^{k-N}\,du,
\]
which is a bijection between the bases
\[
\{u^k\}_{0\le k<N}
\quad\text{and}\quad
\{u^{-j}\,du\}_{1\le j\le N}.
\]
Its inverse is therefore
\[
\iota_{\omega_{\f}}^{-1}\left(\frac{du}{u^{k+1}}\right)
=
u^{N-1-k}.
\]
Applying \(\id\otimes \iota_{\omega_{\f}}^{-1}\) to the Casimir tensor
\eqref{eq:hhy-casimir-N}, we obtain the rank-two Weil operator
\begin{equation}
\label{eq:hhy-weil-N}
\WeilOp
=
\sum_{k=0}^{N-1} u_\one^k \otimes u_\two^{N-1-k}.
\end{equation}
This is exactly the truncated divided-difference kernel of order \(N\). Note
that in this case there is only one point in the support of \(D\), so no
Chinese-remainder assembly is needed.

\subsection{Explicit Szeg\H{o} Kernels}
\label{subsec:hhy-explicit-szego}

We now verify the main theorem explicitly for this example by deriving the Szeg\H{o} kernel in the local coordinate $u = 1/t$ and computing its principal part along $D = N[\infty_t]$. We will see that although the full Szeg\H{o} kernel depends on the choice of root $\zeta_i$, its principal part along $D$ does not.

For each root $\zeta_i$ of $\rho$, define the M\"obius coordinate
\[
\lambda_i = \frac{1}{t - \zeta_i}.
\]
This coordinate has the property that it sends the point $t = \zeta_i$ to $\lambda_i = \infty$, and sends the point $\infty_t$ (where $t = \infty$, i.e., $u = 0$) to $\lambda_i = 0$.

Under this coordinate change, the line bundle $\LL_i = \OO(-[\zeta_i])$ becomes $\OO(-[\infty])$ in the $\lambda_i$-chart, which is trivial on the affine open set $\lambda_i \neq \infty$. Let $e_i = 1$ denote the natural trivializing frame on this chart. By the general construction of the Szeg\H{o} kernel on the projective line (Section \ref{subsec:cauchy-kernel-p1}), the normalized Szeg\H{o} kernel in the $\lambda_i$-coordinate is the standard Cauchy kernel:
\[
S_{\LL_i} = \frac{d\lambda_{i, \two}}{\lambda_{i, \two} - \lambda_{i, \one}} \, e_i(\one) \otimes e_i(\two)^{-1}.
\]

We now rewrite this kernel in terms of the original $t$-coordinate, and then in terms of the local uniformizer $u = 1/t$ at $\infty_t$.

First, express $\lambda_i$ as a function of $u$. Since $t = 1/u$, we have
\[
t - \zeta_i = \frac{1}{u} - \zeta_i = \frac{1 - \zeta_i u}{u},
\]
so
\[
\lambda_i = \frac{u}{1 - \zeta_i u}.
\]

Next, compute the differential $d\lambda_i$ by the quotient rule:
\[
d\lambda_i = d\left( \frac{u}{1 - \zeta_i u} \right) = \frac{du \cdot (1 - \zeta_i u) - u \cdot (-\zeta_i du)}{(1 - \zeta_i u)^2} = \frac{du}{(1 - \zeta_i u)^2}.
\]



Therefore, the Szeg\H{o} kernel in the $u$-coordinate is
\begin{equation}\label{eq:Szego_u}
S_{\LL_i} = \frac{1 - \zeta_i u_\one}{1 - \zeta_i u_{\two}} \cdot \frac{du_{\two}}{u_{\two} - u_\one} \, e_i(\one) \otimes e_i(\two)^{-1}.
\end{equation}

\subsection{principal Part of Szeg\H{o} Kernel}
To compute the principal part of $ S_{\LL_i} $ explicitly, we split the expression \eqref{eq:Szego_u} into two terms by rewriting the prefactor:
\[
\frac{1 - \zeta_i u_\one}{1 - \zeta_i u_{\two}} = 1 + \frac{(1 - \zeta_i u_\one) - (1 - \zeta_i u_{\two})}{1 - \zeta_i u_{\two}} = 1 + \frac{\zeta_i (u_{\two} - u_\one)}{1 - \zeta_i u_{\two}}.
\]

Substituting back, the kernel becomes
\[
\begin{aligned}
S_{\LL_i}
&= \left( 1 + \frac{\zeta_i (u_{\two} - u_\one)}{1 - \zeta_i u_{\two}} \right) \cdot \frac{du_{\two}}{u_{\two} - u_\one} \, e_i(\one) \otimes e_i(\two)^{-1} \\
&= \frac{du_{\two}}{u_{\two} - u_\one} \, e_i(\one) \otimes e_i(\two)^{-1} + \frac{\zeta_i \, du_{\two}}{1 - \zeta_i u_{\two}} \, e_i(\one) \otimes e_i(\two)^{-1}.
\end{aligned}
\]
The first term is the standard singular Cauchy kernel. The second term is \textbf{holomorphic in the second variable at $u_{\two} = 0$}, because the denominator $1 - \zeta_i u_{\two}$ evaluates to $1$ at $u_{\two} = 0$ and is nonvanishing on a neighbourhood of $0$.

By the regularity property of the principal part (Lemma \ref{lemma:pp-basic-properties}(iii)), any section that is holomorphic in the second variable along $D$ has zero principal part. Therefore the second term contributes nothing to $\PP_D(S_{\LL_i})$, and we are left with
\[
\PP_D(S_{\LL_i}) = \PP_D\left( \frac{du_{\two}}{u_{\two} - u_\one} \, e_i(\one) \otimes e_i(\two)^{-1} \right).
\]

Applying the exact truncation formula \eqref{eq:exact-truncation} to the Cauchy kernel, we obtain
\[
\PP_D(S_{\LL_i}) = \left( \sum_{k=0}^{N-1} u_\one^k \otimes \frac{du_{\two}}{u_{\two}^{k+1}} \right) \cdot \left( e_i(\one) \otimes e_i(\two)^{-1} \right).
\]

Applying the frame identification $\Upsilon_{e_i}^D$ to strip off the frame tensor yields
\[
\Upsilon_{e_i}^D\left( \PP_D(S_{\LL_i}) \right) = \sum_{k=0}^{N-1} u_\one^k \otimes \frac{du_{\two}}{u_{\two}^{k+1}} = \inc(\CasFin).
\]

Crucially, this expression is completely independent of the index $i$. Although the full Szeg\H{o} kernel $S_{\LL_i}$ depends on the choice of root $\zeta_i$, the $\zeta_i$-dependent terms are all holomorphic along $D$ and are killed in the principal part quotient. This provides an explicit verification of Proposition \ref{prop:frame-coeff-pp}: every acyclic line bundle yields the same principal part, equal to the scalar extension of the finite residue Casimir.

\subsection{Comparison with the principal part.}
In Section~\ref{subsec:hhy-explicit-szego} we have computed the principal part of the Szeg\H{o} kernel explicitly and find
\[
\PP_D(S_{\LL_i})
=
\sum_{k=0}^{N-1} u_\one^k \otimes \frac{du_\two}{u_\two^{k+1}}
=
\inc(\CasFin),
\]
which matches \eqref{eq:hhy-casimir-N}. Consequently,
\[
\PP_D(S_{\LL_i})
=
\frac{\WeilOp}{f(\two)}\,\omega(\two)
\]
with \(f=\rho^{-1}\) and \(\omega=du\), in accordance with the general
relation between principal parts and Weil operators. This direct comparison
provides an explicit verification of the main theorem in the case where the
differential module is not free.

\section{Example III: Elliptic Curves}
\label{sec:ex-elliptic}

Our final example is an elliptic curve, which demonstrates the theory in genus 1. In this case the canonical bundle is trivial, so $\omegaA$ is principal, but the curve is no longer rational, so the Szeg\H{o} kernel is not simply the Cauchy kernel.

Let $X$ be an elliptic curve over $\F_q$ given by the affine Weierstrass equation
\[
y^2 + a_1 x y + a_3 y = x^3 + a_2 x^2 + a_4 x + a_6.
\]
Take $\infty$ to be the origin (point at infinity), so $U = X \setminus \{\infty\}$ is the affine curve. The coefficient ring $A = \F_q[x, y] / (P(x,y))$ has a canonical invariant differential
\[
\omega_X = \frac{dx}{2y + a_1 x + a_3}.
\]
The module $\omegaA = A \cdot \omega_X$ is free of rank one, i.e., principal, because the canonical bundle of an elliptic curve is trivial.

Now let $V = (\alpha, \beta) \neq \infty$ be a closed point on $X_{\Cinf}$. Define the line bundle
\[
\LL_V = \OO(V - \infty).
\]
This has degree 0. Since $g = 1$, acyclic bundles have degree $g - 1 = 0$. The bundle $\LL_V$ is acyclic if and only if $V \neq \infty$.

The normalized Szeg\H{o} kernel for $\LL_V$ is given explicitly by
\[
S_{\LL_V}(\one,\two) = - \left[ C_X(\one, \two) - C_X(\one, V) \right] \omega_X(\two) \cdot \left( e(\one) \otimes e(\two)^{-1} \right),
\]
where
\[
C_X(\one, \two) = \frac{y_\one + y_{\two} + a_1 x_{\two} + a_3}{x_\one - x_{\two}}
\]
is the primitive Cauchy kernel with respect to the coordinate function $x$, and
\[
C_X(\one, V) = C_X(\one, \two) \mid_{\two = V} = \frac{y_\one + \beta + a_1 \alpha + a_3}{x_\one - \alpha}.
\]

Let us verify the key properties:

\begin{enumerate}
\item \textbf{Diagonal residue.} The term $C_X(\one, \two)$ has a simple pole along the diagonal $x_\one = x_{\two}$ with residue 1 (in the sense of the $x$-coordinate). The term $C_X(\one, V)$ is regular along the diagonal. The factor $\omega_X(\two)$ converts the residue in $x$ to the residue of the differential, giving $\res_{\DeltaDiag}(S_{\LL_V}) = 1$.

\item \textbf{Regularity elsewhere.} 
The function $ x_\one - x_\two $ has two simple poles, located at the diagonal $ \one = \two $, and the anti-diagonal (where $\two$ equals the conjugate point of $ \one $).
On the anti-diagonal, the numerator $y_\one + y_{\two} + a_1 x_{\two} + a_3$ vanishes by the standard identity for Weierstrass curves, so $C_X(\one, \two)$ is actually regular there. Thus the only pole of $S_{\LL_V}$ is along the diagonal.

\item \textbf{Values at infinity.} The kernel has the correct behavior at infinity, because the subtraction of $C_X(\one, V)$ cancels the pole that would otherwise appear at $\two = \infty$.
\end{enumerate}

Now consider a finite divisor $D$ whose support avoids both $V$ and $\infty$. Then $C_X(\one, V)$ is regular in the second variable along $D$, so it contributes nothing to the principal part. Therefore
\[
\PP_D(S_{\LL_V}) = - \PP_D\left( C_X(\one, \two) \, \omega_X(\two) \right).
\]
The reproducing property holds:
\[
\res_{\two=\one}^{\two} \left( - C_X \, \tilde{\alpha} \, \omega_X \right) = \tilde{\alpha}(\one)
\]
for any meromorphic function $\tilde{\alpha}$. By the same argument as in Theorem \ref{thm:reproducing-property}, this implies that $\PP_D(S_{\LL_V})$ acts as the identity on sections restricted to $D$. Therefore
\[
\Upsilon_e(\PP_D(S_{\LL_V})) = \inc(\CasFin),
\]
in accordance with the main theorem \eqref{eq:main-theorem}.

This example demonstrates that our theory extends naturally and seamlessly to higher‑genus curves. Although the Szeg\H{o} kernel is no longer a simple rational function, its principal part still realizes the universal Casimir tensor. Notably, the function \(C_X(\one,\two)\) corresponds to the \(G_u\)-function, and \(C_X(\one,\two)-C_X(\one,V)\) corresponds to the \(J_u\)-function that appears in the construction of Anderson generating functions on elliptic curves \cite{GP18}. 
We will explore this connection in greater depth in forthcoming work.

\bibliographystyle{amsplain}
 \bibliography{paper}

\end{document}